\documentclass{ws-aejm}

\newcommand{\gl}{\mathrm{gl}}
\newcommand{\Au}{\mathrm{Aut}}

\newcommand{\ad}{\mathrm{ad}}
\newcommand{\Ad}{\mathrm{Ad}}
\newcommand{\K}{\mathrm{K}}
\newcommand{\T}{\mathrm{Tr}}
\newcommand{\af}{\mathrm{aff}}
\newcommand{\Af}{\mathrm{Aff}}

\newcommand{\G}{\mathcal{G}}

\newcommand{\h}{\mathfrak{h}}

\newcommand{\X}{\mathfrak{X}}
\newcommand{\Rank}{\operatorname{rank}}

\newcommand{\Lnone}{$\mathrm{Lie} \left(n, 1\right)$}
\newcommand{\Li}{$\mathrm{Lie} \left(n, 2\right)$}
\newcommand{\Lnk}{$\mathrm{Lie} \left(n, k\right)$}
\newcommand{\N}{\mathbb{N}}
\newcommand{\R}{\mathbb{R}}
\newcommand{\C}{\mathbb{C}}
\newcommand{\s}{\mathrm{span}}
\newcommand{\ve}{\vert}

\begin{document}

\setcounter{page}{1}

\markboth{Tu T. C. Nguyen and Vu A. Le}{Representation of {\Li} and geometry of $\K$-orbits}

\doi{--} 

\catchline{--}{--}{--}{--}{--}

\title{\uppercase{Representation of Real Solvable Lie Algebras Having 2-dimensional Derived Ideal and Geometry of Coadjoint Orbits of Corresponding Lie Groups}}

\author{TU T. C. NGUYEN}

\address{ Faculty of Mathematics and Computer Science\\            
University of Science -- Vietnam National University\\
Ho Chi Minh City, Vietnam;\\
College of Natural Sciences --
Can Tho University\\
Can Tho City, Vietnam\\
camtu@ctu.edu.vn}

\author{VU A. LE}

\address{Department of Economic Mathematics\\         
University of Economics and Law -- Vietnam National University\\
Ho Chi Minh City, Vietnam\\
vula@uel.edu.vn}

\maketitle

\begin{history}
\end{history}

\begin{abstract}
Let {\Lnk} be the class of all $n$-dimensional real solvable Lie algebras having $k$-dimensional derived ideals. In 2020 the authors et al. gave a classification of all non 2-step nilpotent Lie algebras of {\Li}. We propose in this paper to study representations of these Lie algebras as well as their corresponding connected and simply connected Lie groups. That is, for each algebra, we give an upper bound of the minimal degree of a faithful representation. Then, we give a geometrical description of coadjoint orbits of corresponding groups. Moreover, we show that the characteristic property of the family of maximal dimensional coadjoint orbits of a MD-group studied by K. P. Shum and the second author et al. is still true for the Lie groups considered here. Namely, we prove that, for each considered group, the family of the maximal dimensional coadjoint orbits forms a measurable foliation in the sense of Connes. The topological classification of these foliations is also provided. 
\end{abstract}



\section{Introduction}\label{sec:1}

The problem of classifying all solvable Lie algebras in an arbitrarily large finite dimension is presently unsolved and is generally believed to be unsolvable. All known full classifications terminate at relatively low dimensions or contain a certain specific given property such as classifying solvable Lie algebras with low-dimensional derived ideal introduced immediately below.

Denote by {\Lnk} the class of all $n$-dimensional real solvable Lie algebras having $k$-dimensional derived ideals. Several classifications of {\Lnk} with $k \in \{1,2\}$ have been investigated. More precisely, a complete classification of {\Lnone} and a partial classification of {\Li} were given by Sch\"obel \cite{Sch93} in 1993. Eberlein \cite{Ebe03} in 2003 gave a formal classification of 2-step nilpotent Lie algebras including 2-step nilpotent ones in {\Li}. Janisse \cite{Jan10} in 2010 gave a partial classification of {\Li} over an algebraically closed field. In 2020, the authors et al. \cite{Le20} presented a new approach to the problem of classifying {\Li} and gave a list of non 2-step nilpotent Lie algebras in {\Li}.

This paper is concerned with the issue of studying representations of Lie algebras in {\Li} and can be considered as a continuation of \cite{Le20}.

As is well known, the Ado-Iwasawa theorem \cite{Ado35,Iwa48} asserts that every finite-dimensional Lie algebra has a faithful representation of finite degree. It is well known that the degree of a faithful representation can be arbitrarily large. It can not be arbitrarily small however. Concretely, it is bounded from below by the square root of the dimension of the Lie algebra. Denote by $\mu(\G)$ the minimal degree of a faithful representation of a Lie algebra $\G$. In general, it is very difficult to determine $\mu(\G)$ and it is also hard to obtain suitable estimations for $\mu(\G)$.  

On the other hand, the orbit method, which is also known as the Kirillov theory, establishes a correspondence between irreducible unitary representations of a Lie group and its coadjoint orbits (see \cite{Kir04}). The theory was first introduced by A. A. Kirillov in the early 1960's and remains a useful and powerful tool in representation theory of Lie groups. It is obvious that coadjoint orbits play an important role in the orbit method. So it is worthwhile to study coadjoint orbits.

We consider in this paper the indecomposable Lie algebras listed in \cite{Le20}. Besides constructing faithful representations of small degree for a family of them, we carry out a survey of coadjoint orbits of all corresponding Lie groups. 
We hope that some beautiful properties of MD-groups studied by K. P. Shum and the second author et al. in \cite{Vu-Shum,Vu90,V93,VH09,VHT14} can be generalized for all considered Lie groups of general dimension $n \geq 5$. 

The main results of the paper are as follows. First, we give in Theorem \ref{thm14} the upper bound of $\mu(\G)$ for each Lie algebra $\G$ under consideration. After that, we describe in Theorem \ref{thm20} the geometry of all coadjoint orbits of connected and simply connected Lie groups corresponding to these algebras. Next, Theorem \ref{FormedFoliation} shows that, for each considered group of general dimension $n \geq 5$, the family of the maximal dimensional coadjoint orbits forms a measurable foliation in the sense of Connes in \cite{Con82}. Finally, the topological classification of these foliations is provided in Theorem \ref{TopoTypes} and their Connes' $C^*$-algebras are described in Corrolary \ref{C312}.  

The paper is organized into four sections including this introduction. Section \ref{sec:2} recalls some notions and useful results. All main results are presented and proved in Section \ref{sec:3}. Finally, Section \ref{sec:4} is some concluding remarks.

\section{Preliminaries}\label{sec:2}

Throughout this paper, the notation $\G$ will mean a finite-dimensional solvable Lie algebra over the field $\R$ of real numbers and $\G^*$ will denote the dual space of $\G$. By $G$ we always mean the connected and simply connected Lie group such that its Lie algebra is $\G$. Some more notations are shown below:

\begin{itemlist}
	\item For any $n$-dimensional real vector space $V$ ($0 < n \in \N$), $\gl(V)$ is the Lie algebra of all endomorphisms of $V$, $\Au(V)$ is the Lie group of all automorphisms of $V$.	
	\item With a choosen basis $\{ e_1, \ldots, e_n \}$ in $V$ and a arbitrary vector $v \in V$, the notation $v(x_1, \ldots, x_n)$ means that the coordinates of $v$ with respect to $\{ e_1, \ldots, e_n \}$ are exactly $(x_1, \ldots, x_n)$.
	\item $\R^n$ is the \emph{real abelian Lie algebra} of dimension $n$.
	\item $\af (\R) = \s \{X_1, X_2\}$ is the \emph{real affine Lie algebra} with $[X_1, X_2] = X_2$. This is also the Lie algebra of the group Aff($\R$) of all affine transformations on $\R$.
	\item $\af(\C) = \s \{X_1, X_2, X_3, X_4\}$ is the \emph{complex affine Lie algebra} 
	with non-trivial Lie brackets: $[X_3,X_1] = -X_2, [X_3,X_2] = [X_4,X_1] = X_1, [X_4,X_2] = X_2$. This is also the Lie algebra of both the group Aff($\C$) of all affine transformations on $\C$ and the simply connected cover $\widetilde{\Af(\C)}$ of Aff($\C$).
	\item $\h_{2m+1} = \s \left\lbrace X_1, X_2, \ldots, X_{2m+1} \right\rbrace$ is the \emph{$(2m+1)$-dimensional real Heisenberg Lie algebra} ($m \geq 1$) with non-trivial Lie brackets: $[X_i,X_{m+i}] = X_{2m+1}$ for all $i = 1, \ldots, m$. 	
\end{itemlist}

\subsection{Representations and $\K$-orbits}

In this subsection, we recall some notions and well-known results about representations which will be used later. For more details, we refer the reader to the references \cite{Kir76,Kir04} of A. A. Kirillov.

\begin{definition}
   A \emph{linear representation} (or \emph{representation}) of $\G$ on a vector space $V$ is a Lie algebra homomorphism $\rho: \G \to$ $\gl(V), \, X (\in \G) \mapsto\rho_X \in \gl(V)$. If $V$ is finite-dimensional, then the representation is called \emph{finite-dimensional} and the dimension of $V$ is called the \emph{degree} of this representation. The representation $\rho$ is called \emph{faithful} if it is injective.
\end{definition}

\begin{definition}
   Given a Lie algebra $\G$, then $\mu(\G)$ is defined to be the minimal dimension of $V$ such that $\G$ admits a faithful representation on $V$.
\end{definition}

\begin{remark}\label{rem10} It is easy to check the following assertions.
	\begin{romanlist}[(ii)]
		\item If $\rho: \G \to \gl(V)$ and $\rho': \G \to \gl(V')$ are two representations of $\G$, then the direct sum $\rho \oplus \rho': \G \to \gl(V \oplus V')$ defined by 
		$$(\rho \oplus \rho')_X(v,v')=(\rho_X(v), \rho'_X(v')) \quad \mbox{for all } X \in \G; \,v \in V;\,v' \in V'$$
		is also a representation of $\G$.
		\item If $\rho$ is a finite-dimensional representation of $\G$ and faithful on the centre of $\G$, then the direct sum of $\rho$ and adjoint representation $\ad$ is faithful on $\G$, where $ \ad: \G \to \gl(\G), \, X \mapsto \ad_X$ such that
		\[\ad_X(Y):= [X,Y] \quad \mbox{for all } Y \in \G.\]
	\end{romanlist}
\end{remark}

For every $g \in G$, we denote the internal automorphism associated with $g$ by $A_g : G \to G$, i.e. 
$A_g(x) := g.x.g^{-1}$ for all $x \in G.$ By taking the derivative at the identity, this automorphism induces the map ${A_g}_{*} : \G \to \G$ defined as follows 
\[{A_g}_{*} (X) : = \frac{d}
{{dt}}\left[ {g.\exp \left( {tX} \right).g^{ - 1} } \right]\left| {_{t = 0} } \right.\quad \mbox{for all } X 
\in \G.\]
This map is called \emph{tangent map} of $A_g$. 

\begin{definition}
	\begin{romanlist}[(ii)]
		\item The action
		\[\Ad: G \to \Au(\G), \, g \mapsto \Ad_g := {A_g}_{*} \quad \mbox{for all } g \in G\] 		
		is called the \emph{adjoint representation} of $G$ in $\G$.
		\item The {\it coadjoint representation} (or {\it $\K$-representation}) of $G$ in $\G^*$ is the action  
   	\[\K: G \to \Au(\G^*), \, g \mapsto \K_g \]
   	such that  	\[\left\langle \K_g (F), X \right\rangle : = \left\langle F, \Ad_{g^{-1}}
   	(X) \right\rangle \quad \mbox{for all } g \in G; F \in \G^*; X \in \G.\]    	
	\end{romanlist}
\end{definition}

\begin{definition}
  The \emph{coadjoint orbit} (or $\K$\emph{-orbit}) of $G$ through any $F \in \G^*$ is defined to be
  				$$\Omega_F=\{ \K_g(F) : g \in G \}.$$
\end{definition}

The description of $\K$-orbits of a given Lie group is very interesting. In 1990, the second author introduced a method of describing based on the structure of Lie algebra (see \cite{Vu90}). Namely, for each $F \in \G^*$, we define 
$$\Omega_F(\G)= \left\lbrace  F_X : X \in \G \right\rbrace$$
in which $F_X \in \G^*$ is defined by
 $$\langle F_X, Y \rangle = \langle F, \exp(\ad_X) ( Y) \rangle \quad \mbox{for all } Y \in \G.$$
Obviously, we always have $\Omega_F(\G) \subseteq \Omega_F$
for all $F \in \G^*$. First, we describe the results on $\Omega_F(\G)$. Then, we show some sufficient conditions to obtain the equality $\Omega_F(\G) = \Omega_F$.
 
 Recall that a Lie group $G$ is called to be \emph{exponential}, if and only if the exponential map $\exp: \G \rightarrow G$ is diffeomorphic.

\begin{proposition}[see \cite{Bou89}]\label{pro8}
	The Lie group $G$ is exponential if and only if for each $X \in \G$ the operator $\ad_X$ has no purely imaginary eigenvalues.
\end{proposition}

\begin{proposition}[see \cite{Vu90}]\label{pro6}
  If $G$ is exponential then $\Omega_F(\G) = \Omega_F$ for all $F \in \G^*$.
\end{proposition}

\begin{proposition}[see \cite{Vu90}]\label{pro7}
Let $G$ be a connected Lie group. Assume that $\{ \Omega_F(\G) \}_{F \in \G^*}$ form a partition of $\G^*$ and for all $F' \in \Omega_F$ the sets $\Omega_{F'}(\G)$ are all open (or closed). Then we have $\Omega_F(\G)=\Omega_F$ for all $F \in \G^*$.
\end{proposition}

\begin{definition} Let $\rho$ be a finite-dimensional representation of $\G$. Then the \emph{character} $\chi_\rho$ of $\G$ afforded by $\rho$ is the function given by 
		\[\chi_\rho(X) := \T \left( \exp(\rho(X)) \right) \quad \mbox{for all } X \in \G,\]		
where $\T$ is the trace	operator.
\end{definition}

\subsection{Foliations and measurable foliations}\label{Subsection2.2}

Let $V$ be an $n$-dimensional smooth manifold ($0 < n \in \mathbb{N}$). We always denote its tangent bundle by $TV$, the tangent space of $V$ at $x \in V$ by $T_xV$. Let us recall some notions in the theory of foliations. For more details, we refer the reader to the reference \cite{Con82} of Connes.

\begin{definition}
A smooth subbundle $\mathcal{F}$ of $TV$ is called {\it integrable} if and only if every $x \in V$ is contained in a submanifold $W$ of $V$ such that $T_pW = \mathcal{F}_p; \forall p \in W$.
\end{definition}

\begin{definition}
A {\it foliation} $(V, \mathcal{F})$ is defined by a smooth manifold $V$ and an integrable subbundle $\mathcal{F}$ of $TV$. Then, $V$ is called the {\it foliated manifold} and $\mathcal{F}$ is called the {\it subbundle defining the foliation}. The dimension of $\mathcal{F}$ is also called the {\it dimension of the foliation} $(V, \mathcal{F})$ and $n - \dim \mathcal{F}$ is called the {\it codimension of the foliation} $(V, \mathcal{F})$ in $V$. Each maximal connected submanifold $L$ of $V$ such that $T_xL = \mathcal{F}_x$ $(\forall x \in L)$ is called a {\it leaf} of the foliation $(V, \mathcal{F})$.
\end{definition}	

Each $x \in V$ has always a system of local coordinates $\{U; x_1, \dots, x_n\}$ such that $x \in U$ and for any leaf $L$ with $L \cap U \neq \emptyset$, each connected component of $L \cap U$ (which is called a {\it plaque} of the leaf $L$) is given by the equations
\[x_{k+1}=c_1, \ldots, x_n=c_{n-k},\]
where $k=\dim \mathcal{F} < n$ and $c_1, \ldots , c_{n-k}$ are suitable constants. Each such system $\{U; x_1, \ldots, x_n\}$ is called a {\it foliation chart}.

A $k$-dimensional foliation can be given by a partition of $V$ in a family $\mathcal{C}$ of its $k$-dimensional submanifolds ($k \in \mathbb{N}, 0 < k < n$) if there exists an integrable $k$-dimensional subbundle $\mathcal{F}$ of $TV$ such that each $L \in \mathcal{C}$ is a maximal connected integral submanifold of $\mathcal{F}$. In this case, $\mathcal{C}$ is the family of leaves of the foliation $(V, \mathcal{F})$. Sometimes $\mathcal{C}$ is identified with $\mathcal{F}$ and we say that $(V, \mathcal{F})$ is formed by $\mathcal{C}$.	

\begin{definition}\label{ToPo}
	Two foliations $(V_1, \mathcal{F}_1)$ and $(V_2, \mathcal{F}_2)$ are said to be {\it topologically equivalent} if and only if there exists a homeomorphism $h: V_1 \rightarrow V_2$ such that $h$ sends leaves of $\mathcal{F}_1$ onto those of $\mathcal{F}_2$. The mapping $h$ is called a {\it topological equivalence} of considered foliations.
\end{definition}

\begin{definition}
A submanifold $N$ of the foliated manifold $V$ is called a {\it transversal} if and only if $T_{x}V = T_{x}N \oplus {\mathcal{F}}_{x},\, \forall x \in N$. Thus,
	$\dim N = n - \dim \mathcal{F} = \, \mbox{codim} \mathcal{F}$.	
	A Borel subset $B$ of $V$ such that $B\cap L$ is countable for any leaf $L$
	is called a {\it Borel transversal} to $(V, \mathcal{F})$.
\end{definition}

\begin{definition}
A {\it transverse measure} $\Lambda$ for the foliation $(V, \mathcal{F})$ is $\sigma$-additive map 
	$B \mapsto \Lambda (B)$ from the set of all Borel transversals to [0,
	+$\infty$] such that the following conditions are satisfied:
	\begin{romanlist}[(ii)]
	\item If $\psi$ : $B_{1} \rightarrow B_{2}$ is a Borel bijection and
	$\psi (x)$ is on the leaf of any $x \in B_{1}$, then $\Lambda(B_{1}) = \Lambda(B_{2})$.
	\item $\Lambda(K)< + \infty$ if $K$ is any compact subset of a smooth transversal submanifold of $V$.
	\end{romanlist}
	
	By a {\it measurable foliation} in the sense of Connes \cite{Con82} we mean a foliation
	$(V, \mathcal{F})$ equipped with a transverse measure
	$\Lambda$.
\end{definition}

Let $(V, \mathcal{F})$ be a foliation with $\mathcal{F}$ is
oriented. Then the complement of zero section of the bundle
${\Lambda}^{k}(\mathcal{F})$ ($0 < k = \dim \mathcal{F} < n$) has two
components ${\Lambda}^{k}{(\mathcal{F})}^{-}$ and
${\Lambda}^{k}{(\mathcal{F})}^{+}$.
Let $\mu$ be a measure on $V$ and  $\{ U; x_1, \ldots, x_n \}$
be a foliation chart of $(V, \mathcal{F})$.
Then $U$ can be identified with the direct product $N \times {\Pi}$
of some smooth transversal submanifolds $N$ of $V$ and some plaques
$\Pi$. The restriction of $\mu$ on $U \equiv N \times {\Pi}$ becomes
the product ${\mu}_{N} \times {\mu}_{\Pi}$ of measures ${\mu}_{N}$
and ${\mu}_{\Pi}$ respectively.
Let $X \in C^{\infty}{\bigl({\Lambda}^{k}(\mathcal{F})\bigr)}^{+}$
be a smooth $k$-vector field and ${\mu}_{X}$ be the measure on each leaf $L$
determined by the volume element $X$.

\begin{definition}
The measure $\mu$ is called
	{\it $X$-invariant} if and only if ${\mu}_{X}$ is proportional to
	${\mu}_{\Pi}$ for an arbitrary foliation chart $\{ U; x_{1}, \ldots, x_{n} \}$.
\end{definition}

\begin{definition}
Let ($X, \mu$), ($Y, \nu$) be two pairs where $X,Y \in C^{\infty}{\bigl({\Lambda}^{k}(\mathcal{F})\bigr)}^{+}$ and $\mu, \nu$
	are measures on $V$ such that $\mu$ is $X$-invariant, $\nu$ is $Y$-invariant. Then ($X, \mu$ ), ($Y, \nu$ ) are called {\it equivalent} if and only if $Y = \varphi X$ and $\mu = \varphi \nu$ for some $\varphi \in C^{\infty}(V).$
\end{definition}

\begin{remark}\label{R222}
	There is a bijection between the set of transverse measures for
	$(V, \mathcal{F})$ and the set of equivalent classes of pairs ($X, \mu$), where $X \in C^{\infty}{\bigl({\Lambda}^{k}(\mathcal{F})\bigr)}^{+}$ and $\mu$ is a $X$-invariant measure on $V$. Thus, to prove that $(V, \mathcal{F})$ is measurable, we only need to choose a suitable pair ($X, \mu$) on $V$.
\end{remark}

We close this section by recalling classification of all non 2-step nilpotent Lie algebras in {\Li}. The next section will present our main results concerning with these algebras and Lie groups corresponding to them.

\subsection{Classification of non 2-step nilpotent Lie algebras in {\Li}}

\begin{proposition}[see \cite{Le20}]\label{prop12}
	Assume that $\G$ belongs to {\Li}\,\emph{(}$3 \leqslant n \in \N$\emph{)} and is not 2-step nilpotent. Then we can choose a suitable basis $\{X_1, \ldots, X_n\}$ of $\G$ such that the following assertions hold.
	\begin{enumerate}
		\item\label{part1} Assume that $\G$ is indecomposable.
		\begin{romanlist}[(ii)]
			\item If $n = 3$ then $\G$ is one of the following Lie algebras:
			\begin{romanlist}[(b)]
				\item $\G_{3,1(\lambda)}, 0 < \ve\lambda\ve \leq 1 \colon$ $[X_3, X_1] = X_1$ and $[X_3, X_2] = \lambda X_2$.
				\item $\G_{3,2} \colon$ $[X_3, X_1] = X_1$ and $[X_3, X_2] = X_1 + X_2$.
				\item $\G_{3,3(\lambda)}, \lambda \geq 0 \colon$ $[X_3, X_1] = \lambda X_1 - X_2$ and $[X_3, X_2] = X_1 + \lambda X_2$.
			\end{romanlist}
			
			\item If $n = 4$ then $\G$ is one of the following Lie algebras:
			\begin{romanlist}[(b)]
				\item $\G_{4,1} \colon$ $[X_3, X_1] = X_1$ and $[X_3, X_4] = X_2$.
				\item $\G_{4,2} \colon$ $[X_3, X_2] = X_1$ and $[X_3, X_4] = X_2$.
				\item $\G_{4,3} \colon$ $[X_3, X_2] = X_1$, $[X_4, X_1] = X_1$ and $[X_4, X_2] = X_2$.
				\item $\G_{4,4} = \af (\C)$.	               
			\end{romanlist}
			
			\item If $n = 5+2k \, (k \geq 0)$ then $\G \cong \G_{5+2k}$ with $[X_3, X_4] = X_1$ and
			\[
			[X_3, X_1] = [X_4, X_5] = \dotsb = [X_{4+2k}, X_{5+2k}] = X_2.
			\]
			
			\item If $n = 6+2k \, (k \geq 0)$ then $\G$ is one of the following Lie algebras:
			\begin{romanlist}[(b)]
				\item $\G_{6+2k,1} \colon$ $[X_3, X_1] = X_1$ and 
				\[
				[X_3, X_4] = [X_5, X_6] = \dotsb = [X_{5+2k}, X_{6+2k}] = X_2.
				\]
				\item $\G_{6+2k,2} \colon$ $[X_3, X_4] = X_1$ and 
				\[
				[X_3, X_1] = [X_5, X_6] = \dotsb = [X_{5+2k}, X_{6+2k}] = X_2.
				\]
			\end{romanlist}
		\end{romanlist}
		
		\item If $\G$ is decomposable then it is one of the following Lie algebras:
		\begin{romanlist}[(ii)]
			\item $\af (\R) \oplus \af (\R)$ or its trivial extension.
			\item $\af (\R) \oplus \h_{2m+1}$ or its trivial extension when $n > 2m +3$, $m \geq 1$.
			\item A trivial extension of the Lie algebras listed in Part (\ref{part1}).
		\end{romanlist}
	\end{enumerate}
\end{proposition}

\section{The main result}\label{sec:3}

From now on, $\G$ belongs to {\Li} listed in Proposition \ref{prop12}.

\subsection{Upper bound for minimal degree of a faithful representation}

First, let us consider the indecomposable case. It is easily seen that the centers of $\G_{3,1(\lambda)}$, $\G_{3,2}$, $\G_{3,3(\lambda)}$, $\G_{4,3}$ and $\G_{4,4}$ are all trivial; the Lie algebras $\G_{4,2}$, $\G_{5+2k}$ and $\G_{6+2k,2}$ are $3$-step nilpotent; the remaining algebras, $\G_{4,1}$ and $\G_{6+2k,1}$, are not nilpotent with 1-dimensional center and can be rewritten as only one family.

\begin{remark}\label{rem13}
 Assume that $\G$ belongs to $\{\G_{4,1}, \G_{6+2k,1} (k \geq 0)\}$. Then $\G$ has a basis $\{X_1, \ldots, X_n\}$ ($n$ is an even number greater than or equal to 4) with the following non-trivial Lie brackets
 $$[X_3,X_1]=X_1\;\; ;\;\; [X_{3+2i}, X_{4+2i}]=X_2, \quad \left(  0\leq i \leq \frac{n-4}{2}\right).$$
Moreover, the center of $\G$ is $Z(\G)=\s \{X_2\}$.
\end{remark}

The first result in this paper is explicitly constructing faithful representations of small degree for an infinite family of non nilpotent Lie algebras in {\Li}. Thereby, we obtain upper bound of $\mu(\G)$ for each $\G$ being indecomposable in Proposition \ref{prop12}.

\begin{theorem}\label{thm14}
 Let $\G$ be a non 2-step nilpotent Lie algebra in {\Li}. Assume that $\G$ is indecomposable. Then the following assertions hold.
 	\begin{romanlist}[(ii)]
		\item\label{1-thm14} If $\G$ has trivial center then $\mu(\G) \leq n$.
		\item\label{2-thm14} If $\G$ is nilpotent then $\mu(\G) \leq n+1$.
		\item\label{3-thm14} If $\G$ is not nilpotent and has non-trivial center then $\mu(\G) \leq \frac{3}{2}n+1$.
	\end{romanlist}  
\end{theorem}

\begin{proof}
Item (\ref{1-thm14}) is obvious because in this case, the adjoint representation is faithful. Item (\ref{2-thm14}) follows from the results of J. Scheuneman in \cite{Sch74} because $\G$ is $3$-step nilpotent. For Item (\ref{3-thm14}), $\G$ must be in $\{\G_{4,1}, \G_{6+2k,1} (k \geq 0)\}$, therefore it suffices to show that $\G$ has a faithful representation of degree $(\frac{3}{2}n+1)$. 

We observe that $\G$ admits the basis $\{X_1, \ldots, X_n\}$ $(n=2l, l \geq 2)$ with non-trivial Lie brackets given in Remark \ref{rem13}. The Heisenberg Lie algebra $\h_{ n-1}$ of dimension $n-1 = 2(l-1)+1$  has a basis $\{x_i,y_i,z\}$ such that non-trivial brackets are
$$[x_i,y_i]=z, \quad \left( 1 \leq i \leq l-1 \right).$$

We now define a linear map $\varphi: \G \rightarrow \mathfrak{h}_{n-1}$ by
$$\begin{array}{l l l}
 \varphi (X_1)=0,&\varphi (X_2)=z,&\\
 \varphi (X_{3+2i})=x_{i+1},& \varphi (X_{4+2i})=y_{i+1},&\left( 0 \leq i \leq \frac{n-4}{2}\right).
\end{array}$$
It is easy to check that $\varphi$ is a Lie algebra homomorphism. Note that $Z(\G)=\s \{X_2\}$ and $\varphi (X_2)=z \neq 0$. Therefore, the restriction of $\varphi$ on the center $Z(\G)$ is injective. From \cite{Bur98}, we have $\mu(\h_{n-1})=(l-1) + 2 = \frac{1}{2}n+1$. Moreover, it is well known that $\h_{ n-1}$ has a faithful representation $\pi$ of degree $\frac{1}{2}n+1$ which in terms of the canonical basis of $\R^{\frac{1}{2}n+1}$ is given by
\[ \pi \left( \sum_{i=1}^{l-1}a_ix_i + \sum_{i=1}^{l-1}b_iy_i + cz \right) =
\begin{bmatrix}
0 & a_1 & \cdots &a_{l-1} &c \\
&&&&b_1\\
&&0&&\vdots\\
&&&&b_{l-1}\\
&&&&0\\
\end{bmatrix}, \quad (a_i, b_i,c \in \R).
\]
Then $\pi \circ \varphi$ is a representation of degree $\frac{1}{2}n+1$ of $\G$ that is faithful on $Z(\G)$. By Remark \ref{rem10}, $(\pi \circ \varphi) \oplus \ad$ is a faithful representation of degree $\frac{3}{2}n+1$ of $\G$.
\end{proof}

\begin{remark}
	Of course it is necessary to show sharper results for the cases where the algebras are low-dimensional.
\begin{romanlist}[(ii)]
	\item First of all, it follows from \cite{Gha05} that the relation $\leq$ in Theorem \ref{thm14} is actually an equality for all 3-dimensional algebras.
	\item Secondly, $\G_{4,1}, \G_{4,2}, \G_{4,3}, \G_{4,4}$ coincide $4.3$, $4.1$, $4.9 (b=0)$, $4.12$ repeated in \cite{Gha13}, respectively. Hence $\mu(\G_{4,1})=\mu(\G_{4,2})=4$ and $\mu(\G_{4,3})=\mu(\G_{4,4})=3$.
\item Thirdly, $\G_5$ coincides $A_{5,5}$ repeated in \cite{Gha15}, so we get $\mu(\G_5)=4$.
\item Finally, $\G_{6,2}$ coincides $L_{6,10}$ repeated in \cite{Roj16}, so we get $\mu(\G_{6,2})=5$.
\end{romanlist}	
\end{remark}

Denote by $\mu_Z(\G)$ the minimal degree of a representation of $\G$ that is faithful on the centre $Z(\G)$. Proposition \ref{lem17} gives $\mu_Z(\G)$ for all $\G \in \{\G_{4,1}, \G_{6+2k,1} (k \geq 0)\}$.
 
\begin{proposition}\label{lem17}
Let $\G$ be an indecomposable Lie algebra in {\Li}. If $\G$ is not nilpotent and has non-trivial center then $\mu_Z(\G)= \frac{1}{2}n+1$.
\end{proposition}

\begin{proof}
Obviously, $\G \in \{\G_{4,1}, \G_{6+2k,1} (k \geq 0)\}$. The representation $\pi \circ \varphi$ shown in the proof of Theorem \ref{thm14} is faithful on the centre of $\G$. Its degree is $\frac{1}{2}n+1$.

Now, assume that $\rho: \G \to \gl(V)$ is a representation that is faithful on the centre $Z(\G)=\s \{X_2\}$. We have to show $\dim V \geq \frac{1}{2}n+1$. It is clear that $\rho(X_2) \neq 0$, i.e. there exists a $v \in V\setminus \{0 \}$ such that $\rho(X_2)v \neq 0$. We observe that $v$ and all $\rho(X)v$ for $X \in \G$ generate a submodule $W$ of $V$.

Consider the \textit{evaluation map} $e_v: \G \to W$, $X \mapsto \rho(X)v$. Let $\mathfrak{a}= \ker (e_v)$ and $\mathfrak{b}=$ im$(e_v)$. It is easy to check that $\mathfrak{a}$ is a subalgebra of $\G$, not containing $X_2$. Since the derived algebra of $\G$ is spanned by $X_1$ and $X_2$ (see \cite{Le20}), the derived algebra of $\mathfrak{a}$ contained in of $\G$ is 0 or 1-dimensional. In other words, $\mathfrak{a}$ is a abelian subalgebra or a subalgebra having 1-dimensional derived algebra. We have
\noindent
\begin{equation}
\dim V \geq \dim W \geq \dim \mathfrak{b} = \dim \G - \dim \mathfrak{a}.\label{dim}
\end{equation}

\begin{itemlist}
\item We show $\dim \mathfrak{b} \geq \frac{1}{2}n$:
\end{itemlist}

The number on the right hand side of (\ref{dim}) is minimal if $\dim \mathfrak{a}$ is maximal. However, any maximal abelian subalgebra of $\G$ not containing $X_2$ has dimension $\frac{1}{2}n$. The dimension is also $\frac{1}{2}n$ for any maximal subalgebra not containing $X_2$ with 1-dimensional derived algebra. Hence $\dim \mathfrak{b} \geq \frac{1}{2}n.$

\begin{itemlist}
\item We show $v \notin \mathfrak{b}$, i.e., $\dim V \geq \dim \mathfrak{b} +1 \geq \frac{1}{2}n +1$:
\end{itemlist}

The fact is $X_1 \in \mathfrak{a}$ in case $\dim \mathfrak{a}$ is maximal, moreover, $\mathfrak{a}$ is a maximal subalgebra not containing $X_2$ such that $[X,Y]= \alpha X_1$ $(\alpha \in \R)$ for all $X,Y \in \mathfrak{a}$. Assume that $v \in \mathfrak{b}$, then there exists an $X$ not in $\mathfrak{a}$ such that $\rho(X)v = v$. Now, there must be some $Y \in \mathfrak{a}$ such that $$[X,Y]=\alpha X_1 + \beta X_2 \quad \mbox{ for } \alpha, \beta \in \R; \beta \neq 0.$$
(If not, $[X,Y]=\alpha X_1$ for all $Y \in \mathfrak{a}$ and by maximality of $\mathfrak{a}$ it follows $\s \{X,\mathfrak{a} \} = \mathfrak{a}$, contradicting the assumption that $X$ not in $\mathfrak{a}$). We have
$$\begin{array}{c r c l}
&\rho([X,Y])v  &=&\rho(\alpha X_1 + \beta X_2)v\\
\Longleftrightarrow & \rho(X) \rho(Y)v- \rho(Y) \rho(X)v&=& \alpha \rho(X_1)v+ \beta \rho(X_2)v.
\end{array}$$
By using $\rho(X_1)v=\rho(Y)v=0$, $\rho(X)v=v$ and $\beta \neq 0$, we obtain $\rho(X_2)v=0$. This is a contradiction.
\end{proof}

As regards decomposable case, if $\G$ is the direct sum of two Lie algebras $\G_1$ and $\G_2$, obviously $\mu(\G) \leq \mu(\G_1)+ \mu(\G_2)$. Hence it is easy to obtain an upper bound for decomposable algebras. 

\subsection{The picture of $\K$-orbits}

Let $\G$ be an indecomposable Lie algebra in {\Li} listed in Proposition \ref{prop12}. Recall that there exists only one connected and simply connected Lie group $G$ corresponding to $\G$. We shall denote these corresponding Lie groups by the capitals with the same indices as their Lie algebras, e.g. $G_{3,1(\lambda)}$ is corresponding to $\G_{3,1(\lambda)}$.

Denote by $\{X_1, \ldots, X_n\}$ the basis of $\G$. Let $\{X^*_1, \ldots, X^* _n \}$ be the dual basis of this basis. Then we can identify the dual space $\G^*$ with $\R^n$. For each $X \in \G$, the operator $\ad_X$ can be identified with its matrix in the basis $\{X_1, \ldots, X_n\}$. 

Before formulating the main theorem describing the picture of $\K$-orbits, we will specify which Lie group is exponential and carry out the necessary calculations. Here we will use MAPLE to perform most routine calculations.
 
\begin{proposition}\label{propexp}
	All connected and simply connected Lie groups corresponding to all non 2-step nilpotent Lie algebras in {\Li} being indecomposable, except to $\G_{3,3(\lambda =0)}$ and $\G_{4,4}$, are exponential.  
\end{proposition}

\begin{proof}
	For each $\G$ under consideration, we find all eigenvalues of $\ad_X$ for all $X \in \G$. Table 1 shows the results whereby all $\ad_X$ have no purely imaginary eigenvalues except for $\G_{3,3(\lambda =0)}$ and $\G_{4,4}$. The proof is immediate from Proposition \ref{pro8}. 
\end{proof}

As previously introduced, describing K-orbits $\Omega_F$ is reduced to describing $\Omega_F(\G)$. It is thus necessary to find the image of $\ad_X$ under the exponential map. Also by specific calculations in the cases of fixed dimensions and by induction in the cases of arbitrary dimensions, we obtain the results as shown in Table 2.

We add a comment on computing the characters of $\G$.

\begin{remark}\label{thm18} Let $\G$ be indecomposable in Proposition \ref{prop12}.
\begin{romanlist}[(ii)]
\item By taking trace of $\exp(\ad_X)$ (see Table 2), the character of $\G$ afforded by the adjoint representation is obtained as follows.
\begin{center}
\begin{tabular}{l  ll l l  }\hline 
$\G$ & $\chi_{\ad}(X)$, $X= \sum_{i=1}^n x_i X_i$&&$\G$ & $\chi_{\ad}(X)$, $X= \sum_{i=1}^n x_i X_i$ \\ \hline

$\G_{3,1(\lambda)}$ & $e^{x_3}+e^{\lambda x_3}+1$ &&$\G_{3,2}$ & $2e^{x_3}+1$ \\

$\G_{3,3(\lambda)}$& $2e^{\lambda x_3} \cos x_3+1$&&$\G_{4,1}$ & $e^{x_3}+3$ \\

$\G_{4,2}$ & 4&&$\G_{4,3}$ &$2e^{x_4}+2$  \\

$\G_{4,4}$ &$2e^{x_4}  \cos x_3+2$&&$\G_{5+2k}$ &$5+2k$\\

$\G_{6+2k,1}$ &$e^{x_3}+5+2k$&&$\G_{6+2k,2}$ &$6+2k$\\

 \hline

\end{tabular}
\end{center} 

\item The character of $\G$ afforded by $(\pi \circ \varphi) \oplus \ad$ (see the proof of Theorem \ref{thm14}) is $\chi_{(\pi \circ \varphi) \oplus \ad} = \chi_{\ad} + \frac{1}{2}n+1$ because for each $X \in \G$, $(\pi \circ \varphi)(X)$ identified with its matrix in the canonical basis of $\R^{\frac{1}{2}n+1}$ is a strictly upper triangular matrix.
\end{romanlist}
	 
\end{remark}

\begin{table}[!htp]
\tbl{Eigenvalues of $\ad_X$}
{\begin{tabular}{ l  l  l  l }  \toprule
 $\G$  & $\ad_X$, $X= \sum_{i=1}^n x_i X_i \,(x_i \in \R)$  &  Eigenvalues of $\ad_X$&\\ \hline
  &&&\\   
 \multicolumn{4}{l}{For $n =3, 4$: The results are obtained by a direct computation.}\\
  &&&\\ 
 $\G_{3,1(\lambda)}$ &$\begin{bmatrix}x_3&0&-x_1\\0&\lambda x_3&-\lambda x_2\\0&0&0 \end{bmatrix}$&$0,\; x_3,\; \lambda x_3$&\\ 
 
$\G_{3,2}$ & $\begin{bmatrix}x_3&x_3&-x_1-x_2\\0& x_3&- x_2\\0&0&0 \end{bmatrix}$ &$0,\; x_3 \mbox{ (multiplicity } 2) $&\\ 
 
$\G_{3,3(\lambda)}$ &$ \begin{bmatrix} \lambda x_3&x_3&-\lambda x_1-x_2\\-x_3&\lambda x_3&-\lambda x_2+x_1\\0&0&0 \end{bmatrix}$&$0,\; (\lambda \pm i)x_3 $&\\ 
 
$\G_{4,1}$ & $\begin{bmatrix}x_3&0&-x_1&0\\0& 0&- x_4&x_3\\0&0&0&0\\0&0&0&0 \end{bmatrix}$ &$0 \mbox{ (multiplicity } 3), \; x_3$&\\ 

$\G_{4,2}$ & $\begin{bmatrix}0&x_3&-x_2&0\\0& 0&- x_4&x_3\\0&0&0&0\\0&0&0&0 \end{bmatrix}$ &$0 \mbox{ (multiplicity } 4)$&\\
 
$\G_{4,3}$ & $\begin{bmatrix}x_4&x_3&-x_2&-x_1\\0& x_4&0&-x_2\\0&0&0&0\\0&0&0&0 \end{bmatrix} $ &$\begin{cases} 0 \mbox{ (multiplicity } 2), \\ x_4 \mbox{ (multiplicity } 2) \end{cases} $&\\

$\G_{4,4}$ & $\begin{bmatrix}x_4&x_3&-x_2&-x_1\\-x_3& x_4&x_1&-x_2\\0&0&0&0\\0&0&0&0 \end{bmatrix}$ &$\begin{cases} 0 \mbox{ (multiplicity } 2),\\ x_4 \pm ix_3 \end{cases}$&\\ 

&&&\\ 
 \multicolumn{4}{l}{For $n =5+2k, 6+2k$ $(k \geq 0)$: The results are obtained by induction.}\\
 &&&\\ 
$\G_{5+2k}$ &\multicolumn{2}{l}{$\begin{bmatrix}0&0&-x_4&x_3&0&0&0&\cdots&0&0\\x_3&0&- x_1&-x_5&x_4&-x_{5+2.1}&x_{4+2.1}&\cdots&-x_{5+2k}&x_{4+2k}\\0&0&0&0&0&0&0&\cdots&0&0\\\vdots&\vdots&\vdots&\vdots&\vdots&\vdots&\vdots&\ddots&\vdots&\vdots\\0&0&0&0&0&0&0&\cdots&0&0 \end{bmatrix}$}&$0 \mbox{ (multiplicity } 5+2k)$\\

$\G_{6+2k,1}$&\multicolumn{2}{l}{$\begin{bmatrix}x_3&0&-x_1&0&0&0&\cdots&0&0\\0&0&-x_4&x_3&-x_{6+2.0}&x_{5+2.0}&\cdots&-x_{6+2k}&x_{5+2k} \\0&0&0&0&0&0&\cdots&0&0 \\ \vdots&\vdots&\vdots&\vdots&\vdots&\vdots&\ddots&\vdots&\vdots\\0&0&0&0&0&0&\cdots&0&0 \end{bmatrix}$}&$0 \mbox{ (multiplicity } 5+2k), x_3$\\
 
$\G_{6+2k,2}$&\multicolumn{2}{l}{$\begin{bmatrix}0&0&-x_4&x_3&0&0&\cdots&0&0\\x_3&0&-x_1&0&-x_{6+2.0}&x_{5+2.0}&\cdots&-x_{6+2k}&x_{5+2k} \\0&0&0&0&0&0&\cdots&0&0 \\ \vdots&\vdots&\vdots&\vdots&\vdots&\vdots&\ddots&\vdots&\vdots\\0&0&0&0&0&0&\cdots&0&0 \end{bmatrix}$}&$0 \mbox{ (multiplicity } 6+2k)$\\
 
&&\\  \botrule
\end{tabular}}
\end{table}

\begin{table}[!htp]
\tbl{$\exp(\ad_X)$}
  {\begin{tabular}{ l  l  l} \toprule
 $\G$  & $\exp(\ad_X)$ &  \\ \hline
 &&\\
 $\G_{3,1(\lambda)}$ &$\begin{bmatrix}
e^{x_3}  &  0  &  -x_1p \\
0            &  e^{\lambda x_3}  &  -x_2q \\
0            &  0  &  1
\end{bmatrix}$ & where $\begin{cases} p= \frac{e^{x_3}-1}{x_3}, q= \frac{e^{\lambda x_3}-1}{x_3} \quad \mbox{if } x_3 \neq 0;\\
																			p=1, q= \lambda \quad \mbox{if } x_3 = 0 \end{cases}$\\ 

 $\G_{3,2}$ & $\begin{bmatrix}
e^{x_3}  & x_3 e^{x_3}  &  -x_1p-x_2e^{x_3} \\
0            &  e^{x_3}  &  -x_2p \\
0            &  0  &  1
\end{bmatrix}$ & where $\begin{cases} p= \frac{e^{x_3}-1}{x_3} \quad \mbox{if } x_3 \neq 0;\\
																			p=1 \quad \mbox{if } x_3 = 0 \end{cases}$\\

 $\G_{3,3(\lambda)}$ &$ \begin{bmatrix}
e^{\lambda x_3} \cos x_3 &  e^{\lambda x_3} \sin x_3  &  -x_2p-x_1q \\
-e^{\lambda x_3} \sin x_3            &  e^{\lambda x_3} \cos x_3  & x_1p-x_2q \\
0            &  0  &  1
\end{bmatrix} $ & where $\begin{cases} p=  \frac{e^{\lambda x_3} \sin x_3}{x_3}, q= \frac{e^{\lambda x_3} \cos x_3-1}{x_3}\\
																	\quad \mbox{if } x_3 \neq 0;\\
																			p=1, q= \lambda \quad \mbox{if } x_3 = 0 \end{cases}$\\

 $\G_{4,1}$ & $\begin{bmatrix}
e^{x_3}  &  0  &  -x_1p  &  0 \\
0            & 1  &   -x_4  &  x_3 \\
0            &  0  &  1  &0\\
0&0&0&1
\end{bmatrix}$ & where $\begin{cases} p= \frac{e^{x_3}-1}{x_3} \quad \mbox{if } x_3 \neq 0;\\
																			p=1 \quad \mbox{if } x_3 = 0 \end{cases}$\\

 $\G_{4,2}$ & $\begin{bmatrix}
1  &  x_3  &  -\frac{1}{2}x_3x_4-x_2  &  \frac{1}{2}x_3^2 \\
0            & 1  &   -x_4  &  x_3 \\
0            &  0  &  1  &0\\
0&0&0&1
\end{bmatrix}$ & \\ 

$\G_{4,3}$ & $\begin{bmatrix}
e^{x_4}  &  x_3e^{x_4}  &  -x_2p  & - x_1p-x_2x_3q \\
0            & e^{x_4}  &   0  &  -x_2p \\
0            &  0  &  1  &0\\
0&0&0&1
\end{bmatrix}$ & where $\begin{cases} p= \frac{e^{x_4}-1}{x_4}, q= \frac{x_4e^{x_4}-e^{x_4}+1}{x_4^2} \\
																			\quad \mbox{if } x_4 \neq 0;\\
																			p=1, q= \frac{1}{2} \quad \mbox{if } x_4 = 0 \end{cases}$\\ 

$\G_{4,4}$ &$ \begin{bmatrix} e^{x_4} \cos x_3  & e^{x_4} \sin x_3  &  p &q \\  -e^{x_4} \sin x_3  & e^{x_4} \cos x_3  &  -q &p  \\ 0&0&1&0\\ 0&0&0&1	\end{bmatrix} $ &\\
&\multicolumn{2}{l}{\quad where $\begin{cases} p=\frac{1}{x_3^2+x_4^2} \left[ (1-e^{x_4} \cos x_3)(x_1x_3+x_2x_4)+e^{x_4} \sin x_3(x_1x_4-x_2x_3) \right], \\
							q=\frac{1}{x_3^2+x_4^2} \left[ (1-e^{x_4} \cos x_3)(x_1x_4-x_2x_3)-e^{x_4} \sin x_3(x_1x_3+x_2x_4) \right] \\
																			\quad \mbox{if } x_3^2+x_4^2 \neq 0;\\
																			p=-x_2, q= -x_1 \quad \mbox{if } x_3=x_4 = 0 \end{cases}$}\\  

$\G_{5+2k}$ &\multicolumn{2}{l}{ $\left[ 
 \begin{array}{ c  |  c }
\begin{matrix}
1  &  0  &  -x_4  & x_3  &  0 \\
x_3            & 1  &  -\frac{1}{2}x_3x_4-x_1   &  \frac{1}{2}x_3^2-x_5  &  x_4  \\
0            &  0  &  1  &0&0\\
0&0&0&1&0\\
0&0&0&0&1
\end{matrix}   &
\begin{matrix}
0  &  0  &  \cdots  &  0  &  0\\  
-x_{5+2.1}  &  x_{4+2.1}  &  \cdots   & -x_{5+2k}  &  x_{4+2k}\\
0  &  0  &  \cdots  &  0  &  0\\  
0  &  0  &  \cdots  &  0  &  0\\  
0  &  0  &  \cdots  &  0  &  0
\end{matrix}   \\  \hline
[0]  & [I_{2k}]
\end{array}  
 \right]$ }\\ 
 
 $\G_{6+2k,1}$&\multicolumn{2}{l}{$\left[ 
 \begin{array}{ c  |  c }
\begin{matrix}
e^{x_3}  &  0  &  -x_1p &  0 \\
0           & 1  &  -x_4   & x_3    \\
0            &  0  &  1  &0\\
0&0&0&1\\
\end{matrix}   &
\begin{matrix}
0  &  0  &  \cdots  &  0  &  0\\  
-x_{6+2.0}  &  x_{5+2.0}  &  \cdots   & -x_{6+2k}  &  x_{5+2k}\\
0  &  0  &  \cdots  &  0  &  0\\  
0  &  0  &  \cdots  &  0  &  0
\end{matrix}   \\  \hline
[0]  & [I_{2(k+1)}]
\end{array}  
 \right]$ }\\ 
 &\multicolumn{2}{l}{\quad where $\begin{cases} p= \frac{e^{x_3}-1}{x_3} \quad \mbox{if } x_3 \neq 0;\\
																			p=1 \quad \mbox{if } x_3 = 0 \end{cases}$}\\  
																			
 $\G_{6+2k,2}$&\multicolumn{2}{l}{$\left[ 
 \begin{array}{ c  |  c }
\begin{matrix}
1  &  0  &  -x_4  &  x_3 \\
x_3           & 1  &  -\frac{1}{2} x_3x_4-x_1   &\frac{1}{2} x_3^2    \\
0            &  0  &  1  &0\\
0&0&0&1\\
\end{matrix}   &
\begin{matrix}
0  &  0  &  \cdots  &  0  &  0\\  
-x_{6+2.0}  &  x_{5+2.0}  &  \cdots   & -x_{6+2k}  &  x_{5+2k}\\
0  &  0  &  \cdots  &  0  &  0\\  
0  &  0  &  \cdots  &  0  &  0
\end{matrix}   \\  \hline
[0]  & [I_{2(k+1)}]
\end{array}  
 \right]$ }\\ 
 &&\\ \botrule
  \end{tabular} }
  \end{table}

The remaining part of this subsection will be devoted to describing the geometric picture of $\K$-orbits.
 
\begin{theorem}[{The picture of $\K$-orbits}]\label{thm20}
	Let $G$ be a connected and simply connected Lie group such that its Lie algebra $\G$ is indecomposable in {\Li} and is not 2-step nilpotent. Assume that $F(f_1, \ldots, f_n)$ is any fixed element and $X^*(x^*_1, \ldots, x^*_n)$ is an arbitrary element in $\G^* (\equiv \R^n)$ with respect to the basis $\{X^*_1, \ldots, X^* _n\}$. 
	\begin{itemize}
		\item[(i)] If $f_1=f_2=0$, the closed set $\{X^*: x^*_1 = x^*_2 = 0\}$ $\bigl(\equiv \{(0,0)\} \times \R^{n-2} \bigr)$ includes only trivial $\K$-orbits (i.e. 0-dimensional ones) $\Omega_F \equiv \{F\}$.
		
		\item[(ii)]  Conversely, the open set $\{X^* : {x^*_1}^2+{x^*_2}^2 \neq 0\}$ $\bigl(\equiv (\R^2 \setminus \{(0,0)\}) \times \R^{n-2} \bigr)$ is decomposed into non-trivial $\K$-orbits as in each of the cases below.
	
		\begin{enumerate}
			\item For $n = 3; \, G \in \{G_{3,1(\lambda)}; G_{3,2}; G_{3,3(\lambda)}\}$: All non-trivial $\K$-orbits are ones of dimension $2$ and they are half-planes or vertical cylinders as follows
			\[ \Omega_F= \begin{cases} \{ X^*(f_1s, f_2s^{\lambda},t) :  s,t \in \R; s>0 \} \quad \;\; \mbox{in case }G=G_{3,1(\lambda)};\\
				\{ X^*(f_1e^s, (f_1s+f_2)e^s,t) :  s,t \in \R \} \;\; \mbox{in case }G=G_{3,2}
			\end{cases}\]
			and in case $G=G_{3,3(\lambda)}$, by identifying $\G^*$ with $\C \times \R$ we have
			\[ \Omega_F= \{ X^*\bigl((f_1+if_2)e^{\lambda s+is}, t\bigr) \,: \, s,t \in \R \}.\]
			
			\item For $n = 4; \, G \in \{ G_{4,1}; G_{4,2}; G_{4,3}; G_{4,4} \}$
			\begin{itemize}			
				\item[2.1] When $G = G_{4,1}$: All non-trivial $\K$-orbits are $2$-dimensional.
				\begin{enumerate}
					\item[(a)] If $f_1 \neq 0 = f_2$ then $\Omega_F$ is a $2$-dimensional half-plane as follows
										\[\Omega_F= \{ X^*: x^*_2 = 0, \, x^*_4 = f_4; \, f_1 x^*_1 > 0, \, x^*_3 \in \R\}.\]
					\item[(b)] If $f_2 \neq 0$ then $\Omega_F$ is a $2$-dimensional cylinder or $2$-dimensional plane as follows
										\[\Omega_F= \{ X^*: x^*_2 = f_2, \, x^*_1 e^{-\frac{x^*_4}{x^*_2}} = f_1 e^{-\frac{f_4}{f_2}}, \, x^*_3 \in \R\}.\]
				\end{enumerate}
				\item[2.2] When $G = G_{4,2}$: All non-trivial $\K$-orbits are $2$-dimensional.
				\begin{enumerate}
					\item[(a)] If $f_1 = 0 \neq f_2$ then $\Omega_F$ is a $2$-dimensional plane as follows
										\[\Omega_F= \{ X^*: x^*_1 = 0, \, x^*_2 = f_2; \, x^*_3, \, x^*_4 \in \R\}.\]
					\item[(b)] If $f_1 \neq 0$ then $\Omega_F$ is a $2$-dimensional parabolic cylinder as follows
										\[\Omega_F= \{ X^*: x^*_1 = f_1, \, {x^*_2}^2 - 2x^*_1 x^*_4 = {f_2}^2 - 2f_1 f_4, \, x^*_3 \in \R \}.\]
				\end{enumerate}
				
				\item[2.3] When $G = G_{4,3}$: All non-trivial $\K$-orbits are either $2$-dimensional or $4$-dimensional.
				\begin{enumerate}
					\item[(a)] If $f_1 = 0 \neq f_2$ then $\Omega_F$ is a $2$-dimensional half-plane as follows
					\[\Omega_F= \{ X^*: x^*_1 = 0, \, x^*_3 = f_3, \, f_2 x^*_2 > 0, \, x^*_4 \in \R\}.\]
					
					\item[(b)] If $f_1 \neq 0$ then $\Omega_F$ is one of two $4$-dimensional half-spaces, namely it is
					\[\Omega_F= \{ X^*: f_1 x^*_1>0; \, x^*_2, \, x^*_3, \, x^*_4 \in \R \}.\]
				\end{enumerate}
				
				\item[2.4] When $G = G_{4,4} = \widetilde{\Af(\C)}$: There is only one  non-trivial $\K$-orbit which is $4$-dimensional. The unique non-trivial $\K$-orbit is exactly the open set $\{F : f_1^2+f_2^2 \neq 0\}$ $\bigl(\equiv (\R^2 \setminus \{(0,0)\}) \times \R^2 \bigr)$. 	
			\end{itemize}
			
			\item \label{Item3} For $n = 5 + 2k, \, G = G_{5+2k}; \, k \in \N$: All non-trivial $\K$-orbits are either $2$-dimensional or $(4+2k)$-dimensional.
			\begin{enumerate}
				\item If $f_2= 0\neq f_1$ then $\Omega_F$ is a $2$-dimensional plane as follows
				\[\Omega_F= \{ X^*: x^*_1 = f_1, x^*_2 = 0, x^*_j = f_j, \, \forall j = 5, \ldots, 5 + 2k\}.\]
				
				\item If $f_2 \neq 0$ then $\Omega_F$ is $(4 + 2k)$-dimensional, namely it is the hyperplane as follows
				\[\Omega_F = \{ X^*: x^*_2 = f_2, \, x^*_1, \, x^*_3, \, x^*_4, \, x^*_j \in \R; \, \forall j = 5, \ldots, 5 + 2k\}.\]				
			\end{enumerate}
			
			\item For $n = 6 + 2k, \, G \in \{G_{6+2k,1}; G_{6+2k,2} \}; \, k \in \N$: All non-trivial $\K$-orbits are either $2$-dimensional or $(4+2k)$-dimensional.
			\begin{enumerate}
				\item If $f_2= 0 \neq  f_1$ then $\Omega_F$ is $2$-dimensional.	
				\begin{itemize}
					\item When $G=G_{6+2k,1}$, $\Omega_F$ is a $2$-dimensional half-plane as follows
					\[\Omega_F= \{ X^*: f_1 x^*_1 > 0, x^*_2 = 0, x^*_j = f_j;\, \forall j = 4, \ldots, 6+2k\}.\]
					
					\item When $G=G_{6+2k,2}$, $\Omega_F$ is a $2$-dimensional plane as follows
					\[\Omega_F= \{ X^*: x^*_1 = f_1, x^*_2 = 0, x^*_j = f_j; \, \forall j = 5, \ldots, 6+2k\}.\]
				\end{itemize}
				
				\item \label{Item4b} If $f_2 \neq 0$ then $\Omega_F$ is $(4 + 2k)$-dimensional.
				\begin{itemize}
					\item When $G=G_{6+2k,1}$, $\Omega_F$ is a $(4+2k)$-dimensional cylinder or $(4+2k)$-dimensional plane as follows 
					\begin{multline*}
						\Omega_F= \{ X^*: x^*_2 = f_2, \, x^*_1 e^{-\frac{x^*_4}{x^*_2}} = f_1 e^{-\frac{f_4}{f_2}}; \\ x^*_3, \, x^*_j \in \R, j = 5, \ldots, 6+2k \}.
					\end{multline*}
					
					\item When $G=G_{6+2k,2}$, $\Omega_F$ is a $(4+2k)$-dimensional parabolic cylinder as follows
					\begin{multline*}
						\Omega_F= \{ X^*: x^*_2 = f_2, \, {x^*_1}^2 - 2 x^*_2 x^*_4 = {f_1}^2 - 2 f_2 f_4; \\ x^*_3, \, x^*_j \in \R, j = 5, \ldots, 6+2k \}.
					\end{multline*}
				\end{itemize}
			\end{enumerate}
		\end{enumerate}
	\end{itemize}  
\end{theorem}

\begin{proof}
	From Proposition \ref{propexp} and Proposition \ref{pro6}, it follows that $\Omega_F=\Omega_F(\G)$ except $\G=\G_{3,3(\lambda=0)}$ and $\G=\G_{4,4}$. Recall that
	\[\Omega_F(\G) = \left\lbrace  F_X : X \in \G \right\rbrace \subset \G^*.\] 
	We consider a general element
	$$F_X=\sum_{i=1}^n x^*_i X^*_i = X^*(x^*_1, \ldots, x^* _n) \in \G^* \equiv \R^n.$$
	By a direct computation, we obtain $(x^*_1, \ldots, x^* _n) = (f_1, \ldots, f_n) \times \exp(\ad_X)$. Note that $\exp (\ad_X)$ is shown in Table 2.
	\begin{itemlist}
		\item Case $G=G_{3,1(\lambda)}$, we have
		$$\begin{cases} x^* _1= f_1e^{x_3}\\
			x^* _2=f_2 e^{\lambda x_3}\\
			x^* _3=-f_1 \frac{x_1(e^{x_3}-1)}{x_3}-f_2 \frac{x_2(e^{\lambda x_3}-1)}{x_3}+f_3
		\end{cases}$$
		if $x_3 \neq 0$ and $(x^* _1,x^* _2,x^* _3)=(f_1,f_2,-f_1x_1- \lambda f_2x_2+f_3)$ if $x_3=0$. Hence $\Omega_F=\{ F\}$ if $f_1=f_2=0$. In other words, each point $F$ with $f_1=f_2=0$ is a 0-dimensional $\K$-orbit. Otherwise, $f_1^2+f_2^2 \neq 0$, 
		$x^* _3$ always runs over a real line. Set $s=e^{x_3}$, then
		$$\Omega_F=\{ X^*(f_1s, f_2s^{\lambda},t) :  s,t \in \R; s>0 \}.$$
		\begin{itemlist}
			\item If $f_1 \neq 0$ and $f_2=0$, $\Omega_F$ is the half-plane $\{X^*: x^*_2 = 0, \, f_1 x^*_1 > 0, \, x^*_3 \in \R\}$.
			\item The same conclusion can be drawn for $f_1 =0, f_2 \neq 0$, namely $\Omega_F$ is the half-plane $\{X^*: x^*_1 = 0, \, f_2 x^*_2 > 0, \, x^*_3 \in \R\}$.
			\item If $f_1 \neq 0$ and $f_2 \neq 0$, $\Omega_F$ is the 2-dimensional vertical cylinder 
			\[\{X^*: \lambda \ln |x^*_1| - \ln |x^*_2| = \lambda \ln |f_1| - \ln |f_2|, \, f_1 x^*_1 > 0, \, f_2 x^*_2 > 0, \, x^*_3 \in \R \}.\]
		\end{itemlist}
		
		\item Case $G=G_{3,2}$, we have
		$$\begin{cases} x^* _1= f_1e^{x_3}\\
			x^* _2=f_1 x_3 e^{x_3}+f_2e^{x_3}\\
			x^* _3=-f_1 \frac{x_2x_3e^{x_3}+x_1e^{x_3}-x_1}{x_3}-f_2 \frac{x_2(e^{x_3}-1)}{x_3}+f_3
		\end{cases}$$
		if $x_3 \neq 0$ and $(x^* _1,x^* _2,x^* _3)=(f_1,f_2,-f_1(x_1+x_2)- f_2x_2+f_3)$ if $x_3=0$. Hence each point $F$ with $f_1=f_2=0$ is also a 0-dimensional $\K$-orbit. Otherwise, $f_1^2+f_2^2 \neq 0$, 
		$x^* _3$ runs over a real line. Set $s=x_3$, then
		$$\Omega_F=\{ X^*(f_1e^s, (f_1s+f_2)e^s,t) : s,t \in \R \}.$$
		\begin{itemlist}
			\item If $f_1 = 0$ and $f_2 \neq 0$, $\Omega_F$ is the half-plane $\{X^*: x^*_1 = 0, \, f_2 x^*_2 > 0\}$.
			\item If $f_1 \neq 0$, $\Omega_F$ is the 2-dimensional vertical cylinder
			\[\{X^*: x^*_2 - x^*_1 \ln |x^*_1| = f_2 - f_1 \ln |f_1|, \, f_1 x^*_1 > 0, \, x^*_3 \in \R\}.\]
		\end{itemlist}
		
		\item Case $G=G_{3,3(\lambda)}$, we have
		$$\begin{cases} x^* _1=f_1e^{\lambda x_3} \cos x_3-f_2 e^{\lambda x_3} \sin x_3\\
			x^* _2=f_1 e^{\lambda x_3} \sin x_3+f_2 e^{\lambda x_3} \cos x_3\\
			x^* _3=f_1  \frac{-x_2e^{\lambda x_3} \sin x_3-x_1 e^{\lambda x_3} \cos x_3+x_1}{x_3} + f_2 \frac{x_1e^{\lambda x_3} \sin x_3-x_2 e^{\lambda x_3} \cos x_3+x_2}{x_3}+f_3
		\end{cases}$$
		if $x_3 \neq 0$ and $(x^* _1, x^* _2, x^* _3)=(f_1,f_2,-f_1(\lambda x_1+x_2)- f_2(\lambda x_2-x_1)+f_3)$ if $x_3=0$. Hence $\Omega_F(\G_{3,3(\lambda)})=\{F\}$ if $f_1=f_2=0$. Otherwise, $f_1^2+f_2^2 \neq 0$, 
		$x^* _3$ always runs over a real line and by identifying $\G_{3,3(\lambda)}^*$ with $\C \times \R$, we have
		$$\Omega_F(\G_{3,3(\lambda)})= \{ X^*(f_1+if_2)e^{\lambda s+is},t) \,: \, s,t \in \R \}$$
		which is the veritical cylinder in $\G_{3,3(\lambda)}^* \equiv \C \times \R$. For $\lambda > 0$ it is true that $\Omega_F=\Omega_F(\G_{3,3(\lambda)})$. We are left with the task of determining $\Omega_F$ in case $\lambda=0$. In fact, the sets $\Omega_F(\G_{3,3(\lambda=0)})$ form a partition of $\G_{3,3(\lambda=0)}^*$. Further, they are all closed sets. By Proposition \ref{pro7} it is obvious that $\Omega_F=\Omega_F(\G_{3,3(\lambda=0)})$. This means $\Omega_F=\Omega_F(\G_{3,3(\lambda)})$ for every $\lambda \geq 0$.
		
		\item Case $G=G_{4,1}$, we have
		$$\begin{cases} x^* _1=f_1e^{x_3}\\
			x^* _2=f_2\\
			x^* _3=-f_1 \frac{x_1(e^{x_3}-1)}{x_3}-f_2 x_4+f_3\\
			x^* _4=f_2x_3+f_4
		\end{cases}$$
		if $x_3 \neq 0$ and $(x^* _1,x^* _2,x^* _3,x^* _4)=(f_1,f_2,-f_1x_1-f_2x_4+f_3,f_4)$ if $x_3=0$. Hence each point $F$ with $f_1=f_2=0$ is also a 0-dimensional $\K$-orbit. Otherwise, $f_1^2+f_2^2 \neq 0$, 
		$x^* _3$ always runs over a real line and $\Omega_F$ is always contained in the hyperplane $\{X^*: x^*_2 = f_2\}$. Set $s=x_3$, we have a family of 2-dimensional $\K$-orbits
		$$\Omega_F=\{ X^*(f_1e^s,f_2, t,f_2s+f_4) : s,t \in \R \}.$$
		\begin{itemlist}
			\item If $f_1 \neq 0$ and $f_2 = 0$, $\Omega_F$ is just a half-plane $\{x^*_2 = 0, \, x^*_4 = f_4, \, f_1 x^*_1 > 0\}$.
			\item If $f_2 \neq 0$, then $\Omega_F$ is the $2$-dimensional cylinder or $2$-dimensional plane 
			\[\{X^*: x^*_2 = f_2, \, x^*_1 e^{-\frac{x^*_4}{x^*_2}} = f_1 e^{-\frac{f_4}{f_2}}, \, x^*_3 \in \R\}.\]
		\end{itemlist}
		
		\item Case $G=G_{4,2}$, we have
		$$\begin{cases} x^* _1=f_1\\
			x^* _2=f_1x_3+f_2\\
			x^* _3=-f_1 ( \frac{1}{2}x_3x_4+x_2)-f_2 x_4+f_3\\
			x^* _4=\frac{1}{2}f_1x_3^2+f_2x_3+f_4.
		\end{cases}$$
		Hence each point $F$ with $f_1=f_2=0$ is also a 0-dimensional $\K$-orbit. Otherwise, $f_1^2+f_2^2 \neq 0$, 
		$x^* _3$ always runs over a real line. 
		Set $s=x_3$, we also have a family of 2-dimensional $\K$-orbits
		$$\Omega_F= \{ (f_1,f_1s+ f_2,t,\frac{1}{2} f_1s^2+f_2s+f_4) \,: \, s,t \in \R \}.$$
		\begin{itemlist}
			\item If $f_1 = 0$ and $f_2 \neq 0$, $\Omega_F$ is the plane $\{X*: x^*_1 = 0, \, x^*_2 = f_2, \, x^*_3, x^*_4 \in \R\}$.
			\item If $f_1 \neq 0$, then $\Omega_F$ is a $2$-dimensional parabolic cylinder 
			\[\{X^*: x^*_1 = f_1, \, {x^*_2}^2 - 2x^*_1 x^*_4 = {f_2}^2 - 2f_1 f_4, \, x^*_3 \in \R\}.\]
		\end{itemlist}
		
		\item Case $G=G_{4,3}$, we have
		$$\begin{cases} x^* _1=f_1e^{x_4}\\
			x^* _2=f_1x_3e^{x_4}+f_2e^{x_4} \\
			x^* _3=-f_1 \frac{x_2(e^{x_4}-1)}{x_4}+f_3\\
			x^* _4=-f_1 \frac{x_2x_3x_4e^{x_4}+(x_1x_4-x_2x_3)(e^{x_4}-1)}{x_4^2} -f_2 \frac{x_2(e^{x_4}-1)}{x_4}+f_4
		\end{cases}$$
		if $x_4 \neq 0$; otherwise, $x_4 = 0$, $(x^* _1,x^* _2,x^* _3,x^* _4)$ is $$\left( f_1,f_1x_3+f_2,-f_1x_2+f_3,-f_1(\frac{1}{2}x_2x_3+x_1)-f_2x_2+ f_4\right).$$ 
		Hence each point $F$ with $f_1=f_2=0$ is also a 0-dimensional $\K$-orbit. Otherwise, $f_1^2+f_2^2 \neq 0$, 
		$x^* _4$ always runs over a real line. 
		\begin{itemlist}
			\item If $f_1 = 0$ and $f_2 \neq 0$, we have 2-dimensional $\K$-orbits which are the half-planes $\Omega_F= \{X^*: x^*_1 = 0, x^*_3 = f_3, \, f_2 x^*_2 > 0\}$. 
			\item If $f_1 \neq 0$,  then 
			$x^* _2$, $x^* _3$ and $x^* _4$ run over three real lines, while 
			$x^* _1$ runs over a half-line. This means that the open set $\{X^* :  x^*_1 \neq 0\}$ is a union of two 4-dimensional $\K$-orbits which are half-spaces 
			$$\Omega_F= \{ X^*: f_1 x^*_1 > 0; \, x^*_2, x^*_3, x^*_4 \in \R\}.$$
		\end{itemlist}
		
		\item Case $G=G_{4,4}$, by Theorem 3.1 of \cite{Die99}, we have 
		\begin{itemlist}
			\item If $f_1 = f_2 = 0$ then $\Omega_F= \{F\}$ is a $0$-dimensional $\K$-orbit.  
			\item If ${f_1}^2 + {f_2}^2 \neq 0$, then the open set $\{X^* :  {x^*_1}^2 + {x^*_2}^2 \neq 0\} \equiv \R^2 \setminus \{(0,0)\} \times \R^2$ is the single 4-dimensional $\K$-orbit. 
		\end{itemlist}
		\item Case $G=G_{5+2k}$, we have
		$$\begin{cases} x^* _1=f_1+f_2x_3\\
			x^* _2=f_2\\
			x^* _3= -f_1x_4-f_2(\frac{1}{2}x_3x_4+x_1) +f_3 \\
			x^* _4=f_1x_3+f_2(\frac{1}{2}x_3^2-x_5) +f_4\\
			x^* _5=f_2x_4+f_5\\
			x^* _i= \begin{cases} -f_2x_{i+1}+f_i, \; \mbox{ if } i \mbox{ is even}\\
				f_2x_{i-1}+f_i, \;\mbox{ if } i \mbox{ is odd} \end{cases} \;\; \mbox{for } 6 \leq i \leq 5+2k.
		\end{cases}$$
		Clearly, each point $F$ with $f_1=f_2=0$ is a 0-dimensional $\K$-orbit. Otherwise, $f_1^2+f_2^2 \neq 0$, the coordinates $x^* _3$ and $x^* _4$ always run over two real lines. So we have
		\begin{itemlist}
			\item If $f_1 \neq 0$ and $f_2 = 0$, then $\Omega_F$ is a $2$-dimensional plane defined by the equations $x^* _1 = f_1, x^* _2 = 0, x^* _5 = f_5, \ldots, x^* _{5+2k} = f_{5 + 2k}$.
			\item If $f_2 \neq 0$, $\Omega_F$ is the hyperplane given by the equation $x^* _2 = f_2$ in $\G^* \equiv \R^{5+2k}$.
		\end{itemlist}
		
		\item Case $G=G_{6+2k,1}$, we have
		$$\begin{cases} x^* _1=f_1e^{x_3}\\
			x^* _2=f_2\\
			x^* _3= \begin{cases} -f_1\frac{x_1(e^{x_3}-1)}{x_3}-f_2x_4 +f_3 \;\, \mbox{ if } x_3 \neq 0\\
				-f_1x_1-f_2x_4+f_3 \qquad \quad \mbox{ if }x_3=0 \end{cases}
			\\
			x^* _4=f_2x_3 +f_4\\
			
			x^* _i= \begin{cases} -f_2x_{i+1}+f_i, \; \mbox{ if } i \mbox{ is odd}\\
				f_2x_{i-1}+f_i, \;\mbox{ if } i \mbox{ is even} \end{cases} \;\; \mbox{for } 5 \leq i \leq 6+2k.
		\end{cases}$$
		Hence each point $F$ with $f_1=f_2=0$ is also a 0-dimensional $\K$-orbit. Otherwise, $f_1^2+f_2^2 \neq 0$, then  $\Omega_F$ is always contained in the hyperplane $\{X^*: x^*_2 = f_2\}$. Namely, by setting $s=x_3$, we have
		\begin{multline*}
			\Omega_F= \{ X^*: x^*_1 = f_1 e^s, \, x^*_2 = f_2, \, x^*_4 = f_2 s + f_4, \, x^*_j = f_2 t_j + f_j; \\ x^*_3, \, s, \, t_j \in \R, j = 5, \ldots, 6+2k \}.
		\end{multline*}
		Now, by removing the parameters we get
		\begin{itemlist}
			\item If $f_1 \neq 0 = f_2$ then $\Omega_F$ is just a $2$-dimensional half-plane
			\[\{X^*: x^*_2 = 0, \, x^*_4 = f_4, \, f_1 x^*_1 > 0, x^*_j = f_j; \, j = 3, 5, \ldots, 6+2k\}.\]
			\item If $f_2 \neq 0$ then 
			$\Omega_F$ is a $(4+2k)$-dimensional cylinder or $(4+2k)$-dimensional plane given by equations: 
			$x^*_2 = f_2, \, x^*_1 e^{-\frac{x^*_4}{x^*_2}} = f_1 e^{-\frac{f_4}{f_2}}$.
		\end{itemlist}
		
		\item Case $G=G_{6+2k,2}$, we have
		$$\begin{cases} x^* _1=f_1+f_2x_3\\
			x^* _2=f_2\\
			x^* _3= -f_1x_4-f_2(\frac{1}{2}x_3x_4+x_1) +f_3 \\
			x^* _4=f_1x_3 +\frac{1}{2}f_2x_3^2+f_4\\
			
			x^* _i= \begin{cases} -f_2x_{i+1}+f_i, \; \mbox{ if } i \mbox{ is odd}\\
				f_2x_{i-1}+f_i, \;\mbox{ if } i \mbox{ is even} \end{cases} \;\; \mbox{for } 5 \leq i \leq 6+2k.  
		\end{cases} $$ 
		Hence each point $F$ with $f_1=f_2=0$ is also a 0-dimensional $\K$-orbit. Otherwise, $f_1^2+f_2^2 \neq 0$, then $\Omega_F$ is always contained in the hyperplane $\{X^*: x^*_2 = f_2\}$. Namely, by setting $s=x_3$, we have
		\begin{multline*}
			\Omega_F= \{ X^*: x^*_1 = f_1 + f_2 s, \, x^*_2 = f_2, \, x^*_4 = f_1 s + \frac{1}{2} f_2 s^2 + f_4, \, x^*_j = f_2 t_j + f_j; \\ x^*_3, \, s, \, t_j \in \R, j = 5, \ldots, 6+2k \}.
		\end{multline*}
		Now, by removing the parameters we get
		\begin{itemlist}
			\item If $f_2= 0 \neq f_1$ then 
			$\Omega_F$ is just a $2$-dimensional plane
			\[\{ x^*_1 = f_1, \, x^*_2 = 0, x^*_j = f_j; j = 3, 4, 5, \ldots, 6+2k\}.\] 
			
			\item If $f_2 \neq 0$ then 
			$\Omega_F$ is a $(4+2k)$-dimensional parabolic cylinder given by the equations $x^*_2 = f_2, \, {x^*_1}^2 - 2x^*_2 x^*_4 = {f_1}^2 - 2f_2 f_4$.
		\end{itemlist}
		The proof is complete.
	\end{itemlist} 
\end{proof}

\begin{remark}
	By looking at the picture of $\K$-orbits of considered Lie groups, we have some geometrical characteristics as follows.
	\begin{romanlist}[(ii)]
		\item Because $G$ is exponential for each Lie group $G$ from  $G_{3,1}$, $G_{3,2}$, $G_{3,3(\lambda > 0)}$, $G_{4,1}$, $G_{4,2}$, $G_{4,3}$, $G_{5+2k}$, $G_{6+2k, 1}$, $G_{6+2k, 2}$, all $\K$-orbits of $G$ are connected, simply connected submanifolds of $\G^* \equiv \R^n$. Moreover, non-trivial $\K$-orbits are all homeomorphic to Euclidean spaces (see \cite{Kir76}). It can be also easily verified by using the picture of $\K$-orbits in Theorem \ref{thm20}. In particular, this property is obvious for $G_{3,3(\lambda = 0)}$ and $G_{4,4} = \widetilde{\Af(\C)}$.
		
		\item For each group $G \in \{G_{3,1}, G_{3,2},  G_{3,3(\lambda)}, G_{4,1}, G_{4,2}, G_{4,4}\}$, there are exactly two types of $\K$-orbits: either $0$-dimensional (trivial) or maximal dimensional. They are called MD-groups in term of D. N. Do \cite{Die99}. Moreover, the family of maximal dimension of each $G$ always forms a measurable foliation in term of Connes \cite{Con82}. This foliation is called MD-{\it foliation} assocated to $G$. In 1990, the second author studied the topology of MD-foliations assocated to all MD-groups of dimension $4$ and desribed the Connes' $C^*$-algebras of these MD-foliations (see \cite{Vu90,V93,Die99}).
		
		\item The only $4$-dimensional case that is not an MD-group is the group $G_{4,3}$. However, since its maximal dimensional $\K$-orbits are exactly two half-spaces it is too simple. Therefore, we do not need to study this group further. 
		
		\item In the general case of the dimension $n \geq 5$, three groups $G_{5+2k}, G_{6+2k, 1}$ and $G_{6+2k,2}$ are not MD-groups. For them, there are exactly three types of $\K$-orbits: either dimension $0, 2$ or maximal dimension. Similar to MD-groups, we will show in the next subsection that the family of maximal dimensional $\K$-orbits of each of the groups $G_{5+2k}, G_{6+2k, 1}$ and $G_{6+2k, 2}$ forms a measurable foliation.
	\end{romanlist}
\end{remark}

\subsection{Foliations formed by the maximal dimensional  $\K$-orbits}

We will establish in this subsection the remaining new results of the paper on the foliations formed by the maximal dimensional $\K$-orbits of each Lie group $G$ of genenral dimension, i.e. $G \in \{G_{5+2k}, G_{6+2k, 1},  G_{6+2k, 2}\}$. 

\begin{theorem}\label{FormedFoliation}
	For any group $G \in \{G_{5+2k}, G_{6+2k, 1},  G_{6+2k, 2}\}$, assume that $\mathcal{F}_G$ is the family of all maximal dimensional $\K$-orbits of $G$ and $V_G : = \cup \{\Omega: \Omega \in \mathcal{F}_G\}$. Then $V_G$ is an open submanifold of $\G^* \equiv \R^n$ and $\mathcal{F}_G$ forms a measurable foliation of dimension $(4+2k)$ on $V_G$ in the sense of Connes \cite{Con82} and this foliation $(V_G, \mathcal{F}_G)$ is called the general {\em MD}-foliation ({\em GMD}-foliation, for short) associated to $G$.	
\end{theorem}

\begin{proof}
	The proof is analogous to the case of MD-groups in \cite{Vu90,V93,VH09,VHT14}. Recall that $4+2k$ is always the maximal dimension of each $\K$-orbit of the family $\mathcal{F}_G$. For any $G$ in the considered set, we first need build a suitable differential system $S_G$ of rank $4+2k$ on the manifold $V_G$ such that each $\K$-orbit $\Omega$ from ${\mathcal{F}}_{G}$ is a maximal connected integrable submanifold corresponding to this system. As the next step, we have to show that the Lebegues measure is invariant for a smooth polyvector field $\X$ of degree $4+2k$ such that it generates $S_G$.
	\vskip0.5cm
	\noindent {\it Step 1: Prove that $(V_G, \mathcal{F}_G)$ is a foliation for each $G \in \{G_{5+2k}, G_{6+2k,1}, G_{6+2k,2}\}$}
	\vskip0.2cm
	Firstly, we consider the simplest case $G = G_{5+2k}$. Recall that for any $F(f_1,\ldots, f_n) \in \G^*$ ($n = 5+2k$), by Item (\ref{Item3}) of Theorem \ref{thm20}, the $\K$-orbits $\Omega_F$ belongs to $\mathcal{F}_G$ if and only if $f_2 \neq 0$, namely, $\Omega_F$ is a hyperplane as follows
	\noindent	
	\begin{equation}
		\Omega_F = \{ X^*: x^*_2 = f_2, \, x^*_1, \, x^*_3, \, x^*_4, \, x^*_j \in \R; \, \forall j = 5, \ldots, 5 + 2k\}.\label{E5} 
	\end{equation} 
	It is clear that $V_G = \cup \{\Omega: \Omega \in \mathcal{F}_G\} \equiv \R \times (\R\setminus \{0\}) \times \R^{n-2}$ and it is open submanifold of $\G^* \equiv \R^n$. In particular, the dimension of submanifold $V_G$ is exaxlty $\dim \G^* = n = 5+2k$. On  
	$V_G$ we consider the differential system 
	\[S_{5+2k} : = \{\X_1 = \frac{\partial}{\partial x^*_1}, \, \X_j = \frac{\partial}{\partial x^*_j}; \,  j = 3, 4, 5, \ldots, 5+2k \}\]
	contained $4+2k$ vector fields. Obviously, $\Rank(S_{5+2k})=4+2k = n - 1$ and $\X_1, \, \X_j$ are smooth over $V_G$, $j = 3, 4, 5, \ldots 5+2k$. Because each $\Omega$ from $\mathcal{F}_G$ is a hyperplane given by Eq. (\ref{E5}) so it is obvious a maximal connected integrable submanifold corresponding to $S_{5+2k}$. In other words, $S_{5+2k}$ generates $\mathcal{F}_{G}$ and $(V_G,\, \mathcal{F}_{G})$ is a $(4+2k)$-dimensional foliation for the case $G = G_{5+2k}$.
	
	Nextly, we consider the case $G = G_{6+2k, 1}$ (i.e. $n = 6+2k$). For any $F(f_1,\ldots, f_n) \in \G^*$, by Item (\ref{Item4b}) of Theorem \ref{thm20}, the $\K$-orbits $\Omega_F$ belongs to $\mathcal{F}_G$ if and only if $f_2 \neq 0$. Furthermore, we have
	\noindent	
	\begin{equation}
		\Omega_F=\{ X^*: x^*_2 = f_2, \, x^*_1 e^{-\frac{x^*_4}{x^*_2}} = f_1 e^{-\frac{f_4}{f_2}}, \, x^*_j \in \R;\, j = 3, 5, \ldots, 6+2k\}.\label{E61} 
	\end{equation}
	It is clear that $V_G = \cup \{\Omega: \Omega \in \mathcal{F}_G\} \equiv \R \times (\R\setminus \{0\}) \times \R^{n-2}$ and it is open submanifold of $\G^* \equiv \R^n$. On the open submanifold 
	$V_G$ 
	we consider the differential system $S_{6+2k, 1} : = \{\X_1, \, \X_3,\, \X_j; \,  j = 5, \ldots, 6+2k \}$ contained $4+2k$ vector fields as follows 
	\[\begin{cases} \X_1:=x^*_1 \frac{\partial}{\partial x^*_1}+ x^*_2 \frac{\partial}{\partial x^*_4}, \, \X_3: =\frac{\partial}{\partial x^*_3}, \\
		\X_j: =\frac{\partial}{\partial x^*_j}; \; j = 5, \ldots, 6+2k.  
	\end{cases}\]
	Obviously, $\Rank(S_{6+2k,1})=4+2k = n - 2$ and $\X_1, \, \X_3, \, \X_j$ are smooth over $V_G$, $j = 5, \ldots 6+2k$. We will show that $S_{6+2k,1}$ produces $\mathcal{F}_{G}$, i.e. show that each $\K$-orbit $\Omega$ from ${\mathcal{F}}_{G}$ is a maximal connected integrable submanifold corresponding to $S_{6+2k,1}$.
	
	Now we consider vector fields  
	$\X_3 : = \frac{\partial}{\partial x^*_3}$ and $\X_j: = \frac{\partial}{\partial x^*_j}$ ($j = 5, \ldots, 6+2k$). Clearly, the flows (i.e. one-parameter subgroups) of them are determined as follows
	\[\theta^{\X_3}_{x^*_3-f_3}: F \mapsto \theta^{\X_3}_{x^*_3-f_3}(F):=(f_1, \, f_2, \, x^*_3, \, f_4, \, f_5, \ldots, \, f_{6+2k}),\]
	\[\theta^{\X_j}_{x^*_j-f_j}: F \mapsto \theta^{\X_j}_{x^*_j-f_j}(F):=(f_1, \ldots, f_{j-1}, \, x^*_j, \, f_{j+1}, \ldots, f_{6+2k});\, j = 5, \ldots, 6+2k.\]
	
	\noindent Next, we consider $\X_1:= x^*_1 \frac{\partial}{\partial x^*_1}+ x^*_2 \frac{\partial}{\partial x^*_4}$. Assume that 
	\[\varphi: s \mapsto \varphi(s)=\bigl(x^*_1(s), \, \ldots, \, x^*_{6+2k}(s), \bigr); s \in \R\]
	is an integral curve of $\X_1$ passing through $F=\varphi(0)$. 
	Hence, we have
	\begin{align}
		\varphi'(s)={\X_1}_{\varphi(s)} \Leftrightarrow & \sum_{j=1}^{6+2k}x^{*'}_j(s)\frac{\partial}{\partial x^*_i} 
		= x^*_1 \frac{\partial}{\partial x^*_1}+ x^*_2 \frac{\partial}{\partial x^*_4} \nonumber\\
		\Leftrightarrow &  \begin{cases}
			x^{*'}_2(s) = 0, \, x^{*'}_3(s) = 0, \, x^{*'}_j(s)= 0; j = 5, \ldots, 6+2k\\
			x^{*'}_1(s)= x^*_1(s)\\
			x^{*'}_4(s)= x^*_2(s)
		\end{cases} \nonumber\\
		\Leftrightarrow &  \begin{cases}
			x^*_2, x^*_3, x^*_j \, \mbox{ are constant functions}, \, j = 5, \ldots, 6+2k\\
			x^{*}_1(s)= constant \times e^s \\ x^{*}_4(s)= x^*_2 s + constant. \label{X2} 
		\end{cases}
	\end{align}
	Combining with condition $F=\varphi(0)$, Eq. (\ref{X2}) gives us
	\[x^*_1= f_1 e^s, \, x^*_4= x^*_2 s + f_4, \, x^*_3 = f_3, \, x^*_j=f_j; \, j = 5, \ldots, 6+2k.\]
	Therefore, the flow of $\X_1$ is
	\[\theta^{\X_1}_{x^*_1-f_1}: F \mapsto \theta^{\X_1}_{x^*_1-f_1}(F):=(f_1 e^s, \, f_2, \, f_3, \, f_2 s + f_4, \, f_5, \dots, \, f_{6+2k}).\]
	By setting $\theta=\theta^{\X_{6+2k}}_{x^*_{6+2k}-f_{6+2k}}\circ \ldots \circ \theta^{\X_5}_{x^*_5-f_5}\circ\theta^{\X_3}_{x^*_3-f_3}\circ\theta^{\X_1}_{x^*_1-f_1}$, we~have
	\begin{align*}
		\theta(F) = & \, \theta^{\X_{6+2k}}_{x^*_{6+2k}-f_{6+2k}}\circ \ldots \circ \theta^{\X_5}_{x^*_5-f_5}\circ\theta^{\X_3}_{x^*_3-f_3}\circ\theta^{\X_1}_{x^*_1-f_1}(F)\\
				= &\, X^*(x^*_1, x^*_2, x^*_3, x^*_4, x^*_5, \ldots, x^*_{6+2k})
	\end{align*}
	where $x^*_3, x^*_j \in \R;\, j = 5, \ldots, 6+2k$ and 
	\[x^*_1= f_1 e^s, \, x^*_ 2 = f_2 , \, x^*_4 = f_2 s + f_4; s \in \R.\] 
	Removing the parameters we get
	\[\{\theta(F) = X^*: x^*_2 = f_2, \, x^*_1 e^{-\frac{x^*_4}{x^*_2}} = f_1 e^{-\frac{f_4}{f_2}}, \, x^*_3, x^*_j \in \R;\, j = 5, \ldots, 6+2k\} \equiv \Omega_F.\] 
	In other words, $\Omega_F$ is a maximal connected integrable submanifold corresponding to $S_{6+2k,1}$. Therefore $S_{6+2k,1}$ generates $\mathcal{F}_{G}$ and $(V_G, \mathcal{F}_{G})$ is a $(4+2k)$-dimensional foliation for the case $G = G_{6+2k,1}$. 
	
	Finally, we consider the case $G = G_{6+2k, 2}$ (i.e. $n = 6+2k$). For any $F(f_1,\ldots, f_n) \in \G^*$, by Item (\ref{Item4b}) of Theorem \ref{thm20}, the $\K$-orbits $\Omega_F$ belongs to $\mathcal{F}_G$ if and only if $f_2 \neq 0$. Furthermore, we have
	\[\Omega_F= \{ X^*: x^*_2 = f_2, \, {x^*_1}^2 - 2 x^*_2 x^*_4 = {f_1}^2 - 2 f_2 f_4; \\ x^*_3, \, x^*_j \in \R, j = 5, \ldots, 6+2k \}.\]
	Of course, $V_G = \cup \{\Omega: \Omega \in \mathcal{F}_G\} \equiv \R \times (\R\setminus \{0\}) \times \R^{n-2}$ and it is open submanifold of $\G^* \equiv \R^n$. On the open submanifold 
	$V_G$ 
	we consider the differential system $S_{6+2k, 2} : = \{\X_1, \, \X_3,\, \X_j; \,  j = 5, \ldots, 6+2k \}$ contained $4+2k$ vector fields as follows 
	\[\begin{cases} \X_1:=x^*_2 \frac{\partial}{\partial x^*_1}+ x^*_1 \frac{\partial}{\partial x^*_4}, \, \X_3: =\frac{\partial}{\partial x^*_3}, \\
		\X_j: =\frac{\partial}{\partial x^*_j}; \; j = 5, \ldots, 6+2k.  
	\end{cases}\]
	By calculations quite similar to the case of group $G_{6+2k,1}$, we can verify that $\Rank(S_{6+2k,1})$ is equal to $4+2k = n - 2$ and $S_{6+2k,2}$ produces $\mathcal{F}_{G}$. Therefore $(V_G, \mathcal{F}_G)$ is a foliation on the open submanifold $V_G$ and the proof of the first step is complete.
	\vskip0.5cm
	\noindent {\it Step 2: Prove that the foliation $(V_G,\mathcal{F}_{G})$ is measurable in the sense of Connes for any $G$ from $G_{5+2k},\, G_{6+2k,1}, \,G_{6+2k,2}$}
	\vskip0.2cm
	We now turn to the second step of the proof. Namely, we have to show that the foliation $(V_G,\mathcal{F}_{G})$ is measurable in the sense of Connes. By Remark \ref{R222}, we only need to choose a suitable pair ($\X_G, \mu _G$) on $V_G$ where $\X_G$ is a smooth polyvector field of degree $4+2k$ defined on $V_G$, $\mu _G$ is a measure on $V_G$ such that $\X$ generates $S_G$ and $\mu _G$ is $\X _G$-invariant. 
	We choose $\mu_G$ to be exactly the Lebegues measure on $V_G$ for any $G \in \{G_{5+2k}, \, G_{6+2k, 1}, \, G_{6+2k,2}\}$.
	For $\X _G$, we set
	\[\X_G : = \begin{cases} \X_1 \wedge \X_3 \wedge \X_4 \wedge \X_5 \wedge \ldots \wedge \X_{5+2k} \qquad {\mbox{if}} \; \; G = G_{5+2k},\\
		\X_1 \wedge \X_3 \wedge \X_5 \wedge \X_6 \wedge \ldots \wedge \X_{6+2k} \qquad {\mbox{if}} \; \; G \in \{G_{6+2k,1}, \, G_{6+2k,2}\}.
	\end{cases}\]
	For any $G \in \{G_{5+2k}, \, G_{6+2k,1}, \, G_{6+2k,2}\}$, it is clear that $\X_G$ is a polyvector field of degree $4+2k$, smooth, non-zero everywhere on $V_G$ and $\X_G$ generates $S_G$. 
	That is, when choosing on $(V_G, \mathcal{F}_G)$ a suitable orientation, then $\X_G \in C^{\infty}{\bigl({\Lambda}^{4+2k}(\mathcal{F})\bigr)}^{+}$. It is obvious that the invariance of the Lebegues measure $\mu_G$ with respect to $\X_G$ is equivalent to the invariance of $\mu_G$ for the $\K$-representation that is restricted to the foliated submanifold $V_G$ in $\G^*$. For any $X(x_1, x_2, x_3, x_4, x_5, \ldots, x_n) \in \G$, direct computation show that Jacobi's determinant
	$J_X$ of differential mapping $\K_{\exp(X)}$ is a constant which depends only on $X$ but do not depend on the coordinates of any point which moves in each $\K$-orbit $\Omega \in \mathcal{F}_G$. In other words, the Lebegues measure $\mu_G$ is $\X_G$-invariant. The proof is complete.
\end{proof}

\begin{theorem}[Description and classification of topological types of GMD-foliations]\label{TopoTypes}
	The topology of  {\em GMD}-foliations in Theorem \ref{FormedFoliation} has the following properties.
	\begin{romanlist}[(ii)]
		\item Three {\em GMD}-foliations $(V_{G_{5+2k}}, \mathcal{F}_{G_{5+2k}})$, $(V_{G_{6+2k,1}}, \mathcal{F}_{G_{6+2k,1}})$ and $(V_{G_{6+2k,2}}, \mathcal{F}_{G_{6+2k,2}})$ determine exactly two types of topology. The single foliation $(V_{G_{5+2k}}, \mathcal{F}_{G_{5+2k}})$ gives the first type and we denote it by $\mathcal{F}_{5+2k}$. Two remaining foliations $(V_{G_{6+2k,1}}, \mathcal{F}_{G_{6+2k,1}}), (V_{G_{6+2k,2}}, \mathcal{F}_{G_{6+2k,2}})$ are topologically equivalent and they give the second type which is denoted by $\mathcal{F}_{6+2k}$. 
		\item The type $\mathcal{F}_{5+2k}$ is a trivial fibration with connected and simply connected fibers on $\R^*: = \R\setminus \{0\}$. 
		\item The type $\mathcal{F}_{6+2k}$ is a trivial fibration with connected and simply connected fibers on $\R^* \times \R$.
	\end{romanlist}	
\end{theorem}

\begin{proof}
	We now prove three assertions of the theorem. 
	\begin{itemlist}
		\item Firstly, we prove Item (i) of the theorem. It is obvious that 
		$(V_{G_{5+2k}}, {\mathcal{F}}_{G_{5+2k}})$ is not topologically equivalent to any foliation from $\{(V_{G_{6+2k,1}}, {\mathcal{F}}_{G_{6+2k, 1}})$, $(V_{G_{6+2k,2}}, {\mathcal{F}}_{G_{6+2k, 2}})\}$ because $V_{G_{5+2k}}$ is $(5+2k)$-dimensional manifold while both of $V_{G_{6+2k, 1}}$ and $V_{G_{6+2k, 2}}$ are $(6+2k)$-dimensional manifolds.
		
		Now we will prove that two foliations $\{(V_{G_{6+2k,1}}, {\mathcal{F}}_{G_{6+2k, 1}}), \, (V_{G_{6+2k,2}}, {\mathcal{F}}_{G_{6+2k, 2}})\}$ are topologically equivalent. Note that 
		\[ V_{G_{6+2k, 1}} \equiv V_{G_{6+2k, 2}} \equiv \R \times \R\setminus \{0\} \times \R^{4+2k}\]  
		and we will denote both of them by $V_{6+2k}$ for convenience. By Definition \ref{ToPo}, to prove the topological equivalence of two foliations which are given on the same foliated manifold $V_{6+2k}$, we need to find a homeomorphism $h$ of $V_{6+2k}$ such that $h$ sends leaves of the first foliation onto those of the second one. \\
		Let $h: V_{6+2k} \rightarrow V_{6+2k}, \, X^*(x^*_1, \ldots, x^*_{6+2k}) \mapsto h(X^*) : = \widetilde{X^*}(\widetilde{x^*_1}, \dots, \widetilde{x^*_{6+2k}})$ which is defined as follows:
		\[ 
		\begin{cases} \widetilde{x^*_1} := x^*_1 + x^*_4, \, \widetilde{x^*_2} := x^*_2, \, \widetilde{x^*_4} := \frac{1}{2 x^*_2} \Bigl[{\bigl(x^*_1 + x^*_4\bigr)}^2 - x^*_1 e^{-\frac{x^*_4}{x^*_2}} \Bigr] ,\\
			\widetilde{x^*_3} := x^*_3, \, \widetilde{x^*_j} := x^*_j; \, j = 5, \ldots, 6+2k.
		\end{cases}
		\]
		It is clear that the considered map $h$ is homeomorphic. 	
		Now, we take an arbitrary leaf $L$ of $(V_{6+2k}, \mathcal{F}_{G_{6+2k,1}})$. In fact, $L$ is determined as follows
		\[L  = \{X^* \in V_{6+2k}: x^*_2 = c_2, \, x^*_1 e^{-\frac{x^*_4}{x^*_2}} = c_1; \, x^*_3, \, x^*_j \in \R; j = 5, \ldots, 6+2k\}\]
		where $c_1, c_2$ are some real constants, $c_2 \neq 0$. For $(V_{6+2k}, \mathcal{F}_{G_{6+2k,2}})$, we consider the leaf $\widetilde{L}$ determined as follows
		\[
		\widetilde{L} = \{ \widetilde{X^*} \in V_{6+2k}: \, \widetilde{x}^*_2 = c_2, \, {\widetilde{x}^* _1}{}^2 - 2 \widetilde{x}^*_2 \widetilde{x}^*_4 = c_1; \, \widetilde{x}^*_3, \, \widetilde{x}^*_j \in \R; j = 5, \ldots, 6+2k\}.
		\]
		Now we look back at the homeomorphism $h: V_{6+2k} \to V_{6+2k}$. It is plain that 
		\begin{align} \notag X^* \in L \Leftrightarrow & \,\,  x^*_2 = c_2, \, x^*_1 e^{-\frac{x^*_4}{x^*_2}} = c_1, \, x^*_3, \, x^*_j \in \R; j = 5, \ldots, 6+2k\\		
			\notag \Leftrightarrow & \,\, \widetilde{x}^*_2 = c_2, \, {\widetilde{x}^* _1}{}^2 - 2 \widetilde{x}^*_2 \widetilde{x}^*_4 = c_1, \, \widetilde{x}^*_3, \, \widetilde{x}^*_j \in \R; j = 5, \ldots, 6+2k\\
			\notag \Leftrightarrow &  \, \, h_1(X^*) = \widetilde{X^*} \in \widetilde{L}.
		\end{align}		 
		In other words, $h(L) = \widetilde{L}$ for any $L$ of $(V_{6+2k}, \mathcal{F}_{G_{6+2k,1}})$. Therefore, two foliations $(V_{6+2k}, \mathcal{F}_{G_{6+2k,1}}), (V_{6+2k}, \mathcal{F}_{G_{6+2k,2}})$ are topological equivalent and they determine the second type $\mathcal{F}_{6+2k}$. Item (i) of the theorem is proven.
		
		\item Now we prove Item (ii) of Theorem \ref{TopoTypes}. Note that all leaves of GMD-foliation $(V_{G_{5+2k}}, \mathcal{F}_{G_{5+2k}})$ are hyperplanes of the form 
		\[ \{X^*: x^*_2 = c; x^*_1, x^*_, x^*_4, x^*_j \in \R; j = 5, \ldots, 5+2k\}\]
		parameterized by the non-zero constant $x^*_2 = c \in \R^*$. Therefore the foliation of the first type $\mathcal{F}_{5+2k}$ come from the following fibration 
		\[p_{5+2k}: V_{G_{5+2k}} \rightarrow \R^*; \, X^*(x^*_1, \, x^*_2, \, \ldots,\, x^*_{5+2k}) \mapsto p_{5+2k} (X^*): = x^*_2.\]
		
		\item Finally, we prove Item (iii) of Theorem \ref{TopoTypes} for the second type $\mathcal{F}_{6+2k}$. Of course, we only need to consider the foliation $(V_{6+2k}, \mathcal{F}_{G_{6+2k,2}})$ instead of both ones of this type.
		
		Let $p_{6+2k}: V_{6+2k} \rightarrow \R^* \times \R$ be the map that defines as follows
		\[X^*(x^*_1, x^*_2, \ldots, x^*_{6+2k}) \mapsto p_{6+2k} (X^*):= (x^*_2, \, {x^*_1} ^2 - 2x^*_2 x^*_4).\]
		It can be verified that $p_{6+2k}$ is a submersion and $p_{6+2k}: V_{6+2k} \rightarrow \R^* \times \R$ is a fibration on $\R^* \times \R$ with connected (and simply connected) fibers. Moreover, each leaf of the foliaion $(V_{6+2k}, \mathcal{F}_{G_{6+2k,2}})$ is exactly one fiber of this fibration. In other world, $(V_{6+2k}, \mathcal{F}_{G_{6+2k,2}})$ comes from this fibration.	
	\end{itemlist}
	The proof is complete.
\end{proof}

As an immediate consequence of Theorem \ref{TopoTypes} and the result of Connes in \cite{Con82}, we have the following result.

\begin{corollary}\label{C312}
	The Connes' $C^*$-algebras of {\em GMD}-foliations in Theorem \ref{TopoTypes} are determined as follows
	\begin{romanlist}[(ii)]
		\item $C^*(\mathcal{F}_{5+2k}) \cong C_0(\R^*) \otimes \mathcal{K} \cong \bigl(C_0(\R) \oplus C_0(\R) \bigr) \otimes \mathcal{K}$
		\item $C^*(\mathcal{F}_{6+2k}) \cong C_0(\R^* \times \R) \otimes \mathcal{K} \cong \bigl(C_0(\R^2) \oplus C_0(\R^2) \bigr) \otimes \mathcal{K}$
	\end{romanlist}	
	where $C_0(X)$ is the algebra of continuous complex-valued functions defined on an arbitrary locally compact space $X$ vanishing at infinity and $\mathcal{K}$ denotes the $C^*$-algebra of compact operators on an (infinite dimensional) separable Hibert space.
\end{corollary}

\section{Conclusion}\label{sec:4}

In the paper, we have considered all non $2$-step nilpotent Lie algebras in the class {\Li} being indecomposable and all connected and simply connected Lie groups corresponding to them. The main results of the paper are as follows: First, we give the upper bound for the minimal degree of a faithful representation for considered algebras. Second, we give the geometric picture of $\K$-orbits of corresponding groups. Third, we proved that, for any considered Lie group $G$ of dimension $n$ in general cases of $n \geq 5$, the family of all maximal dimensional K-orbits forms measurable foliations (in the sense of Connes). We call them GMD-foliations. Finally, the topological classification of all GMD-foliations is given and their Connes' $C^*$-algebras are described. Furthermore, the proposed method can be applied to other Lie groups corresponding to the remaining algebras, i.e. $2$-step nilpotent ones, in {\Li}.


\end{document}